\documentclass[10pt]{amsart}

\usepackage{amsfonts}
\usepackage{amssymb}
\usepackage{graphicx} 
\usepackage{amsmath, amsthm, latexsym}
\usepackage[all]{xy}

\usepackage{url}
\usepackage{hyperref}
\usepackage[margin=1.5in]{geometry}
\usepackage{color}

\numberwithin{equation}{section}

\def \c{\mathbb{C}}
\def \z{\mathbb{Z}}
\def \r{\mathbb{R}}
\def \n{\mathbb{N}}
\def \p{\mathbb{P}}

\def \A{\mathcal{A}}
\def \T{\mathbb{T}}
\def \HH{\mathbb{H}}

\def \BB{\mathcal{B}}
\def \X{\mathfrak{X}}

\def \RR{\mathcal{R}}

\def \GL{\textup{GL}}

\def \SP{\textup{Sp}}
\def \SO{\textup{SO}}

\def \lb{\mathcal{L}}
\def \s{{\underline{w}_0}}
\def \w{{\underline{w}_0}}

\def \deg{\textup{deg}}
\def \dim{\textup{dim}}

\def \conv{\textup{conv}}
\def \gr{\textup{gr}}
\def \Vol{\textup{Vol}}

\def \Proj{\textup{Proj}}
\def \Sym{\textup{Sym}}
\def \Lie{\textup{Lie}}
\def \ker{\textup{ker}}

\def \max{\textup{max}}

\def \ratmap{\dashrightarrow}
\def \Re{\textup{Re}}
\def \Im{\textup{Im}}

\theoremstyle{plain}
\newtheorem{Th}{Theorem}[section]
\newtheorem*{Th*}{Theorem}
\newtheorem{Lem}[Th]{Lemma}
\newtheorem{Prop}[Th]{Proposition}
\newtheorem{Cor}[Th]{Corollary}

\theoremstyle{definition}
\newtheorem{Ex}[Th]{Example}
\newtheorem{example}{Example}
\newtheorem{Def}[Th]{Definition}
\newtheorem{Rem}[Th]{Remark}
\newtheorem{DEFIN}{Definition}

\begin{document}
\title{Integrable systems, toric degenerations and Okounkov bodies}

\author{Megumi Harada}
\address{Department of Mathematics and
Statistics\\ McMaster University\\ 1280 Main Street West\\ Hamilton, Ontario L8S4K1\\ Canada}
\email{Megumi.Harada@math.mcmaster.ca}
\urladdr{\url{http://www.math.mcmaster.ca/Megumi.Harada}}
\thanks{The first author is partially supported by an NSERC Discovery Grant,
an NSERC University Faculty Award, and an Ontario Ministry of Research
and Innovation Early Researcher Award.}

\author{Kiumars Kaveh}
\address{Department of Mathematics\\ School
  of Arts and Sciences \\
University of Pittsburgh\\
301 Thackeray Hall \\ Pittsburgh, PA  15260 \\
U.S.A.}
\email{kaveh@pitt.edu} 
\urladdr{\url{http://www.pitt.edu/~kaveh}}
\thanks{The second author is partially supported by a
Simons Foundation Collaboration Grants for Mathematicians and a National Science Foundation Grant.}

\dedicatory{To Askold Georgevich Khovanskii}

%\subjclass[2010]{Primary: 14xx ; Secondary: 53D20, 14M25}

\date{}

\begin{abstract} 
Let $X$ be a smooth projective variety of dimension $n$ over $\c$ equipped
with a very ample Hermitian line bundle $\mathcal{L}$. In the first part of the paper, 
we show that 
if there exists a toric degeneration of $X$
satisfying some natural hypotheses (which are satisfied in many settings), then there exists
a surjective 
continuous map from $X$ to the special fiber $X_0$ which is a symplectomorphism on an open dense subset $U$.
From this we are then able to construct
a {\it completely integrable system} on $X$ in the sense of symplectic geometry. 
More precisely, we construct a collection of real-valued functions $\{H_1, \ldots, H_n\}$ on $X$ 
which are continuous on all of $X$, smooth on an open dense subset $U$
of $X$, and pairwise Poisson-commute on
$U$. Moreover, our integrable system 
in fact generates a Hamiltonian torus action on $U$. 
In the second part, we show that the toric
degenerations arising in the theory of Newton-Okounkov bodies satisfy all
the hypotheses of the first part of the paper. 
In this case the image of the `moment map' $\mu = (H_1, \ldots, H_n): X
\to \r^n$ is precisely the \textit{Newton-Okounkov body} $\Delta = \Delta(R, v)$
associated to the homogeneous coordinate ring $R$ of $X$, and an
appropriate choice of a valuation $v$ on $R$. 
Our main technical tools come from algebraic geometry, differential
(K\"ahler) geometry, and analysis. Specifically, we use the gradient-Hamiltonian vector field, 
and a subtle generalization of the famous \L{}ojasiewicz gradient inequality for real-valued analytic
functions. Since our construction is valid for a large class of projective varieties $X$, this manuscript
provides a rich source of new examples of integrable systems. We
discuss concrete examples, including elliptic curves, flag varieties of arbitrary
connected complex reductive groups, spherical varieties, and weight
varieties. 
\end{abstract}

\maketitle

\setcounter{tocdepth}{1}
\tableofcontents

\section*{Introduction} \label{sec-intro}
A (completely) integrable system on a symplectic manifold is a
Hamiltonian system which admits a maximal number of first integrals, 
also called ``conservation laws.'' A first integral is a function
which is constant along the Hamiltonian flow; when there are a maximal
number of such, then one can describe the integral curves of the
Hamiltonian vector field implicitly by setting the first integrals
equal to constants. In this sense an integrable system is very
well-behaved. For a recent overview of this subject, see
\cite{PelayoVuNgoc} and its extensive bibliography. The theory of
integrable systems in symplectic geometry is rather dominated by
specific examples (e.g. ``spinning top,'' ``Calogero-Moser system,'' ``Toda
lattice'').  The main contribution of this manuscript, 
summarized in Theorems A, B below, 
is a construction of an integrable system on (an open dense subset of) a
wide class of complex projective varieties. 
Moreover, we additionally show that
the integrable system in fact generates a Hamiltonian
torus action on an open dense subset of the variety.
Much is true about Hamiltonian torus actions that is false for
integrable systems in general.

Our construction
represents a significant expansion of the possible areas of
application of the theory of integrable systems and 
Hamiltonian torus actions, including: 
complex algebraic geometry,
symplectic topology, representation theory, and Schubert calculus, to
name just a few. We therefore expect our manuscript to be of
wide-ranging interest. Moreover, we use a rather broad mix of
techniques coming from algebraic geometry, differential (K\"ahler)
geometry, and analysis, in order to prove our main results. With this
in mind, we have attempted to make our exposition accessible to a
broad audience. 

Throughout, our algebraic varieties are over $\c$. 

We begin with a definition; for details see e.g. \cite{Audin96}. Let
$(M,\omega)$ be a symplectic manifold of real dimension $2n$. Let
$\{f_1, f_2, \ldots, f_n\}$ be smooth functions on $M$. Then (by
  slight abuse of terminology, see Section~\ref{sec-prelim-part1}) we call the
functions $\{f_1, \ldots, f_n\}$ a {\bf (completely) integrable system}
on $M$ if they satisfy the following conditions: 
\begin{itemize} 
\item they pairwise \emph{Poisson-commute}, i.e. $\{f_i, f_j\} =
0$ for all $i, j$, where $\{\cdot, \cdot\}$ denotes the Poisson
bracket on $\mathcal{C}^{\infty}(M)$ induced from $\omega$, and 
\item they are \emph{functionally independent},
i.e. their derivatives $df_1, \ldots, df_n$ are linearly independent
almost everywhere on $M$. 
\end{itemize} 
In fact, below we will state and use a slightly
generalized version of this definition (cf. Definition~\ref{def-int-system}). 
% {\BL Note that the terminology concerning (completely) integrable systems 
% is not entirely consistent in the literature; sometimes  
% the whole collection $\{f_1, \ldots, f_n\}$ is called a completely integrable system
% (cf. for example \cite{Audin97}), 
% whereas at other times, a Hamiltonian system - for instance $(M,\omega, f_1)$ - is called
% completely integrable if it has $n$ functionally independent first integrals, i.e. one has the collection 
% $\{f_1, \ldots, f_n\}$ as above (cf. for example \cite{Audin96}).  Under some additional assumptions, the 
% $\{f_1, \ldots, f_n\}$ are also the 
% \emph{action variables (or action coordinates)} in a set of so-called \emph{action-angle coordinates} for the integrable system (and can be interpreted as components of a moment map).
% In this manuscript, in an attempt to simplify terminology we follow \cite{Audin97} and refer to the whole collection $\{f_1, \ldots, f_n\}$ as a 
% ``(completely) integrable system'' instead of ``action variables for a (completely) integrable system''. }

We recall two examples which may be familiar to researchers outside of
symplectic geometry. 

\begin{example} \label{example:intro toric}
  A smooth projective toric variety $X$ is a symplectic manifold,
  equipped with the pullback of the standard Fubini-Study form on
  projective space. The (compact) torus action on $X$ is in fact
  Hamiltonian in the sense of symplectic geometry, and its moment map
  image is precisely the polytope corresponding to $X$. The torus has
  real dimension $n = \frac{1}{2} \dim_{\mathbb{R}}(X)$, and its $n$
  components form an integrable system on $X$.
\end{example}

\begin{example} \label{example:intro GC}
  Let $X = \GL(n,\c)/B$ be the flag variety of nested subspaces in
  $\c^n$. For $\lambda$ a regular highest weight, consider the usual
  Pl\"ucker embedding $X \hookrightarrow \mathbb{P}(V_\lambda)$ where
  $V_\lambda$ denotes the irreducible representation of $\GL(n,\c)$
  with highest weight $\lambda$. Equip $X$ with the
  Kostant-Kirillov-Souriau symplectic form coming from its
  identification with a coadjoint orbit of
  $U(n,\c)$, which meets the positive Weyl chamber at precisely
  $\lambda$. Then Guillemin-Sternberg \cite{G-S} build an integrable
  system, frequently called the \textit{Gel'fand-Cetlin integrable
    system}, on
  $X$ by viewing this coadjoint orbit as a subset
  of the space of Hermitian $n \times n$ matrices. This integrable
  system is intimately related to the well-known Gel'fand-Cetlin
  basis for the irreducible representation $V_\lambda$
  \cite{G-C}. 
\end{example}

Example~\ref{example:intro GC} in fact motivated this paper. Several years ago, we 
had the idea to use the toric degeneration results for flag varieties
given in Gonciulea-Lakshmibai \cite{GoLa96}, Kogan-Miller \cite{KM}, and Caldero \cite{Caldero}, in order 
to construct integrable
systems for
flag varieties of general reductive groups, thus generalizing the
Gel'fand-Cetlin integrable system constructed by Guillemin and
Sternberg. Allen Knutson also had very similar
ideas, according to his post on MathOverflow in response to a
question from David Treumann \cite{KnTr-MO}. (His post is dated January
2010, but he apparently had the ideas already in the mid-90's.)
More formally, a connection
between toric degenerations  and the Gel'fand-Cetlin integrable system was established in detail in 2010 in 
\cite{NNU}, using Ruan's technique of gradient-Hamiltonian
flows. Their ideas inspired us to prove the results contained in the
present paper, which in fact deals not only with flag varieties of
general reductive groups, but with a much more general class of
projective varieties. 

Motivated by the above examples, we make the following definition of a
(completely) integrable system on a projective variety
which differs slightly from some other definitions in
the literature (for more details see Section~\ref{sec-prelim-part1}).
\begin{DEFIN} \label{def-1-intro}
Let $X$ be {an algebraic} variety of dimension $n$. Let $\omega$ be a symplectic form on 
the smooth locus of $X$. We call a collection of real-valued functions $\{F_1, \ldots, F_n\}$ on $X$ a {\bf completely
integrable system} (or an {\bf integrable system}) on $X$ if the following are satisfied:
\begin{itemize}
\item[(a)] $F_1, \ldots, F_n$ are continuous on all of $X$.
\item[(b)] There exists an open dense subset $U$ contained in the
smooth locus of $X$ such that $F_1, \ldots, F_n$ are differentiable
on $U$, and the differentials $dF_1, \ldots, dF_n$ are linearly
independent on $U$. 
\item[(c)] $F_1, \ldots, F_n$ pairwise Poisson-commute on $U$. 
\end{itemize}
\end{DEFIN}

Recall that a special case of an integrable system is given by (the
components of) a moment map for a Hamiltonian torus action on
a symplectic manifold (see Section~\ref{sec-prelim-part1} for
details). Motivated by this, in this manuscript we call the map $\mu
:= (F_1, \ldots, F_n): X \to \r^n$, constructed from the functions
$F_i$ above, the \textbf{moment map of the integrable system}. 

The main purpose of this paper is to construct integrable systems, in
the sense of the definition above, on a large class of projective
algebraic varieties. 
Our main tool is a \emph{toric
degeneration}, in a sense we make precise
below. Let $X$ be an $n$-dimensional projective variety. 

\begin{DEFIN} \label{def-2-intro}
We call an algebraic family $\pi: \X \to \c$ a {\bf toric degeneration of $X$} if the following hold:
\begin{itemize}
\item[(1)] $\pi: \X \to \c$ is a flat family of irreducible varieties. In particular, each fiber is an irreducible reduced scheme.
\item[(2)] The family is trivial over $\c^* = \c \setminus \{0\}$ with fiber isomorphic to $X$; {more precisely,  
there exists a fiber-preserving isomorphism
$\rho: X \times \c^* \to \pi^{-1}(\c^*)$ from the trivial fiber bundle
$X \times \c^* \to \c^*$ to $\pi^{-1}(\c^*)$.}   In other words, $\X
\setminus X_0$ is equipped with a $\c^*$-action, and the map $\pi$ is
$\c^*$-equivariant (where $\c^*$ acts on the base $\c$ by
multiplication). 
\item[(3)] There is an action of $\T = (\c^*)^n$ on the special fiber $X_0$ giving it the 
structure of a toric variety. 
\end{itemize} 
\end{DEFIN}

For the definition of a flat family (of schemes) see \cite[III.9]{Hartshorne}.  
In a flat family all the fibers $X_t = \pi^{-1}(t)$ have the same dimension, and more generally the same Hilbert polynomial 
(\cite[Theorem III.9.9 and Corollary III.9.10]{Hartshorne}). In fact the Hilbert polynomial being the same for all the fibers is equivalent to flatness. 
We note that flatness of a family $\pi: \X \to \c$, in the case when $\X$ is a variety, is a mild condition (\cite[Proposition III.9.7]{Hartshorne}).

We now describe the contents of this paper in more detail. The
paper has two parts. In Part 1 we use a toric degeneration of $X$, satisfying some additional hypotheses explained below,
to construct an integrable system on $X$. In Part 2 we show that this construction can be applied to 
the toric degenerations arising from the theory of Newton-Okounkov bodies (also called Okounkov bodies), as explained in \cite{Anderson10}.
As a consequence, we obtain integrable systems whose moment images are the corresponding Newton-Okounkov bodies. 

First we present precise statements of the
  hypotheses and the result in our main
  theorem in Part 1 (Theorem A below). We require the
  toric degenerations to satisfy some additional properties as follows. 
Let $\pi: \X \to \c$ be a toric degeneration of $X$
{in the sense above and fix the isomorphism $\rho: X
  \times \c^* \to \pi^{-1}(\c^*)$. Note that $\c^*$ acts on $X \times
  \c^*$ by acting standardly on the second factor, and the isomorphism
  $\rho$ then induces 
 a $\c^*$-action on $\X \setminus X_0$}.  We assume the following:
\begin{itemize}
\item[(a)] The family $\X$ is smooth away from the singular locus of the zero fiber $X_0$. 
In particular, the projective variety $X$ is smooth. 
\item[(b)] The family $\X$ is embedded in $\p^N \times \c$ as an algebraic subvariety for some projective space $\p^N$ such that the following hold: 
\begin{itemize} 
\item the map
$\pi$ is the restriction to $\X$ of the standard projection $\p^N \times \c \to \c$ to the second factor, and
\item the action of $\T = (\c^*)^n$ on $X_0$ extends to a linear action of 
$\T$ on $\p^N \times \{0\}$.
\end{itemize}
\end{itemize} 
Now let $\omega$ be a 
{fixed} K\"ahler 
structure on $X$, the general fiber of $\X$ (note $X$ is smooth by
assumption (a) above). In our constructions, we also require that $\p^N \times \c$ (and hence $\X$)
is equipped with a
K\"ahler form which is compatible with $\omega$ in a sense we
now explain. 
Namely, let $\Omega$ be a K\"ahler form on $\p^N$  which is
a constant multiple of a Fubini-Study K\"ahler form (see Remark \ref{rem-toric-int-system-multiple-FS}) and equip $\p^N
\times \c$ with the product K\"ahler structure {given by $\Omega$ and
the standard K\"ahler structure on $\c$.}
We denote by $\tilde{\omega}$ (respectively $\omega_t$) the restriction of this product form to the
smooth locus of the family $\X$ (respectively the fiber {$X_t:=\pi^{-1}(t)$}). 
{For $t \neq 0$, denote by $\rho_t$ the restriction of the isomorphism $\rho: X \times \c^* \to \X \setminus X_0$ to 
$X \times \{t\}$; this is an isomorphism of varieties between $X$ and
the fiber $X_t$}. Also let $T = (S^1)^n$ denote the
compact torus, considered as a subgroup of $\T
= (\c^*)^n$ in the standard fashion. 
With this notation in place we can state the
  following assumptions on our toric degenerations. 
\begin{itemize}
\item[(c)] The pull-back $\rho_1^*(\omega_1)$ {to $X$}
  coincides with the original K\"ahler form $\omega$ on $X$. 
\item[(d)] The K\"ahler form $\Omega$ on $\p^N$ is $T$-invariant;
  {in particular, $\omega_0$ is a $T$-invariant 
K\"ahler form on the toric variety $X_0$.} 
\end{itemize}

The following is the main result of Part~\ref{part-1} (Theorem~\ref{th-A}). 
\begin{Th*}[A] \label{th-A-intro}
Let $X$ be a smooth $n$-dimensional projective variety and let $\omega$ be a K\"ahler structure on $X$.
Suppose that there exists a toric degeneration $\pi: \X \to \c$ of $X$,
in the sense of Definition \ref{def-2-intro}, satisfying
the above conditions (a)-(d). Then:
\begin{itemize}
\item[(1)] There exists a surjective continuous map $\phi: X \to X_0$ which is a 
symplectomorphism restricted to a dense open subset $U \subset X$.
\item[(2)] There exists a completely integrable system $\mu = (F_1, \ldots, F_n)$ on $(X, \omega)$, in the sense of 
Definition \ref{def-1-intro} above, such that its moment image $\Delta$ coincides with the moment image of $(X_0, \omega_0)$ (which is a polytope). 
\item[(3)] Let $U \subset X$ be the open dense subset of $X$ from (1). Then 
the integrable system $\mu=(F_1, \ldots, F_n)$ generates a Hamiltonian
torus action on $U$, and the inverse image {$\mu^{-1}(\Delta^\circ)$ of
the interior of $\Delta$ under the moment map
$\mu: X \to \r^n$ of the integrable system}
lies in the open subset $U$.
\end{itemize}
\end{Th*}
 
We note that: 
\begin{itemize}
\item[(i)] For Theorem (A) part (1), we do not need 
to assume that $X_0$ is a toric variety and that the K\"ahler form $\omega_0$ is torus-invariant. 
\item[(ii)] Theorem (A) still holds (with the same proof) if the special fiber $X_0$ is a union of toric varieties and hence possibly 
reducible as a variety. In this case the moment image will be a union of convex polytopes.
\end{itemize}

We now explain our main result in Part 2, where we apply 
Theorem A to the toric degenerations associated to valuations and 
Newton-Okounkov bodies. 
The theory of Newton-Okounkov bodies is newer than that of toric
degenerations and integrable systems, and was initiated by Okounkov \cite{Okounkov1, Okounkov2} 
and further developed by Kaveh and Khovanskii \cite{KKh1, KKh2} and also independently
by Lazarsfeld and Mustata \cite{LM}. The notion of a Newton-Okounkov body associated to a projective variety far generalizes the notion of Newton polytope of a projective toric variety.

Let $X$ be a smooth $n$-dimensional projective variety
equipped with a very ample line bundle $\lb$ and let $R$ denote the corresponding
homogeneous coordinate ring.  Let $v: \c(X) \setminus \{0\} \to \z^n$ be a valuation with
one-dimensional leaves (see Definition~\ref{def-valuation}) on the space
of rational functions $\c(X)$ on $X$, and 
let $S = S(R)$ denote the corresponding value semigroup in
$\n \times \z^n$ (see Equation~\eqref{eq:definition S}).
We also let $\Delta = \Delta(R)$ denote the \textbf{Newton-Okounkov body} (also called Okounkov body)
corresponding to $R$ and $v$ (Definition~\ref{def-Ok-body}). 
Fix a Hermitian structure $H$ on $\lb$ and let $\omega$ denote the associated K\"ahler structure on $X$, 
induced from the Kodaira embedding to $\p(H^0(X,\lb)^*)$ and the Fubini-Study form on $\p(H^0(X,\lb)^*)$ determined by $H$. 
 With this notation in place we may state the main result of Part 2
(Theorem~\ref{th-B}). 

\begin{Th*}[B]  \label{th-B-intro}
Let $(X, \omega)$ and $S=S(R)$ be as above, and suppose in
addition that the semigroup $S = S(R)$ is
finitely generated. 
\begin{itemize}
\item[(1)] There exists an integrable system 
$\mu = (F_1, \ldots, F_n): X \to \r^n$ on $(X, \omega)$ in the sense of Definition \ref{def-1-intro}. Moreover,
the image of $\mu$ coincides with the Newton-Okounkov body $\Delta = \Delta(R)$.
\item[(2)] Let $U \subset X$ denote the open dense subset of $X$ mentioned in Definition \ref{def-1-intro}. 
Then the integrable system generates a torus 
action on $U$ and the inverse image under $\mu$ of the interior of the moment polytope
$\Delta$ lies entirely in the open subset $U$.
\end{itemize}
\end{Th*}

Theorem (B) can be regarded as a generalization of the fact that the Newton polytope of a toric variety
coincides with the image of its moment map, i.e. its moment polytope.
In particular, Theorem B addresses a question posed to us by
Julius Ross and by Steve Zelditch: does there exist, in general, a
``reasonable'' map from a (smooth) variety $X$ to its Newton-Okounkov body? At least
under the technical assumption that the value semigroup $S = S(R)$ is
finitely generated, our theorem suggests that the answer is
yes, given by the map $\mu$ in Theorem B. Indeed, in this case we can 
do even better: we have a surjective continuous 
map $\phi$ from $X$ to the toric variety $X_0$ associated to the semigroup $S$, which is a symplectomorphism on  a
dense open subset. (The normalization of $X_0$ is the toric variety associated to the polytope $\Delta$.)

The present manuscript opens many doors. 
We now discuss a small selection of possible
avenues for further investigation.  Firstly, 
in future work we intend to further analyze (and in particular, make
explicit computations for) the concrete examples mentioned in this paper (see Section \ref{sec-examples}), as well as the moduli spaces of flat
connections on Riemann surfaces and Jeffrey-Weitsman integrable system
(\cite{J-W}). Secondly, we intend to apply our results to provide 
explicit geometric interpretations of crystal
bases for representations of connected reductive algebraic groups
through the theory of geometric quantization and Bohr-Sommerfeld
fibers (see e.g. \cite{Burns}) using methods similar to the work of Hamilton and Konno in \cite{HamiltonKonno}. 
Thirdly, it has been suggested to us by Milena Pabiniak and Yael Karshon that Theorem B
can be applied to prove Biran's conjecture on the Gromov width of
projective varieties (cf. \cite{K-T}) for a large class of
examples. Fourthly, in the classical case of the Gel'fand-Cetlin
integrable system of Example~\ref{example:intro GC}, it is well-known
that the action variables given by taking eigenvalues of upper-left $k
\times k$ submatrices (which give continuous but non-smooth functions
on $\GL(n,\c)/B$) can be ``smoothed out'' by instead taking the
coefficients of the characteristic polynomials of the same upper-left
submatrices. It would be interesting to better understand whether this
phenomenon occurs for the action variables on $X$ obtained in this manuscript.

Finally, we discuss the issue of the technical
hypothesis present in Theorem B above, namely, the finite
generation of the semigroup $S=S(R)$. Firstly, in the light of Teissier's work \cite{Te99} which constructs toric degenerations even in the non-finitely-generated setting, 
we expect that the methods of 
this manuscript can be modified to prove statements similar to Theorem B even in the case when the value semigroup $S$ 
is not finitely generated.  
In general, 
given a projective variety $X$ and a valuation $v$, {determining
whether or not the corresponding semigroup $S$ is finitely generated
is a subtle and tricky question.} Indeed, it is not even known
for an arbitrary $X$ whether one can 
always find a valuation $v$ with a corresponding value semigroup $S$
that is finitely generated. (However, it is known that such valuations
exist in many interesting and important examples -- cf. Section
\ref{sec-examples}). {In the papers of Anderson \cite{Anderson10} and
Anderson-Kuronya-Lozovanu \cite{AKL12},}
the authors give sufficient conditions on $X$ such that one can find a
valuation $v$ with finitely generated semigroup $S$, but the relation
between the choice of valuation and its corresponding value semigroup
(and Newton-Okounkov body) is still not completely understood and there are many open questions.

We now briefly sketch the ideas of the proofs of our main results. The essential ingredient in the proof of Theorem A
is the so-called ``gradient-Hamiltonian vector field'' (defined by Ruan \cite{Ruan} and also used by 
Nishinou-Nohara-Ueno \cite{NNU}) on $\X$, where we think of $\X$ as a {K\"ahler} space by embedding 
it into the product of an appropriate projective space with $\c$.
The proof of Theorem B relies on a toric degeneration from $X$ to a toric variety
$X_0$ (the normalization of which is the toric variety corresponding to the polytope $\Delta = \Delta(R)$), 
discussed in detail by Anderson \cite{Anderson10}. Let $\pi:\X \to \c$ denote
the flat family with special fiber $\pi^{-1}(0) \cong X_0$ and
$\pi^{-1}(z) = X_z \cong X$ for $z \neq 0$. Since toric varieties are
integrable systems (see Example~\ref{example:intro toric} above), {the
essential idea is to ``pull back''
the integrable system on $X_0$ to one on $X$ using the
gradient-Hamiltonian vector field. }
The main technicalities which must be overcome in order to make {this
idea} rigorous is to appropriately deal with the singular points of
$\X$ and prove that the $f_i$ thus constructed may be
continuously extended to all of $X$. 
It turns out that, in
order to deal with this issue, we need a subtle generalization of
the famous \L{}ojasiewicz inequality.
Finally, we note that since we use both algebro-geometric and differential-geometric techniques 
in our arguments, we work with both the Zariski and the classical (analytic) topologies on algebraic varieties. We have 
attempted to make clear in our statements which topology is meant. 

\medskip
\noindent \textbf{Acknowledgements.} As {we were preparing 
this manuscript, we learned} 
that Allen Knutson had roughly predicted our result
in the MathOverflow discussion thread mentioned above. We thank Allen
for encouraging us to work out these ideas carefully.  We also thank
Chris Manon and Dave Anderson for useful discussions. 
Finally, we thank
the anonymous referees for useful comments which greatly improved the
paper and for pointing out a short proof for
Proposition~\ref{prop-sing-family}. 

\part{Integrable systems from toric degenerations} \label{part-1}

\section{Preliminaries for Part 1} \label{sec-prelim-part1}
We begin by recalling some background. Let $(M,\omega)$ be a
symplectic manifold. A smooth function $H \in C^\infty(M)$, also
called a Hamiltonian,
defines a vector field $\xi_H$, called the {\it Hamiltonian vector field} (associated to $H$) by the equality 
\begin{equation}
  \label{equ-def-xi_H}
  \omega(\xi_H, \cdot) = dH(\cdot).
\end{equation}
The Hamiltonian vector field defines a differential equation
on $M$, the \emph{Hamiltonian system} associated to
$H$. The equation~\eqref{equ-def-xi_H} also defines in turn a \emph{Poisson bracket} on $C^\infty(M)$ by 
\begin{equation}
  \label{equ-def-Poisson-bracket}
  \{f, g \} := \omega(\xi_g, \xi_f).
\end{equation}
A function $f$ on $M$ is said to be a \emph{first integral} of $H$ if
$\{H, f\} = 0$. 
A Hamiltonian system $(M,\omega,H)$ on a $2n$-dimensional symplectic manifold is
{\it completely integrable} (or {\it integrable} for short) if there exist
$n =
\frac{1}{2} \dim_\r M$ first integrals $H_1, H_2, \ldots, H_n$
which 
\begin{itemize}
\item are functionally independent, i.e., there exists an open dense
  subset $U$ of $M$ such that, for all $x \in U$, the differentials
  $dH_i(x)$ are linearly independent, 
\item pairwise Poisson-commute, i.e., $\{H_i, H_j\} =0$ for all $i,j$.
\end{itemize}
By slight abuse of language one often refers to the collection $\{H_1, H_2, \ldots, H_n\}$ as an
integrable system on $(M,\omega)$. 

There are many interesting and well-studied examples of integrable
systems on symplectic manifolds (cf. \cite{Audin96} 
and references therein). However, as far as we are aware, little
is known about methods for constructing interesting 
integrable systems on a general symplectic manifold $(M, \omega)$. 
Moreover, in many cases, 
a (smooth) set of action variables for an integrable system exists only on an open
dense subset of $M$. 
A famous example is the Gel'fand-Cetlin system on
coadjoint orbits of $U(n)$ described by Guillemin and Sternberg
\cite{G-S}. Here, the eigenvalues of the
upper-left $k \times k$ submatrices give rise to action
variables, but only on an open dense subset (because they are not
smooth everywhere). 
The eigenvalues are, however, \emph{continuous}
everywhere. 
In general, it is also desirable to know that a set 
of smooth functions
$\{H_1, H_2, \ldots, H_{\frac{1}{2}\dim_\r M}\}$ giving an integrable
system (or, even better, action variables for an integrable system) on
an open dense subset of $M$ actually 
extends continuously to all of $M$.

The above discussion motivates the following definition. 

\begin{Def} \label{def-int-system}
Let $X$ be a complex variety of dimension $n$ (and hence of real dimension $2n$). Let $\omega$ be a symplectic form on 
the smooth locus of $X$. We call a collection of real-valued functions $\{F_1, \ldots, F_n\}$ on $X$ a {\it completely
integrable system} (or an {\it integrable system}) on $X$ if the following are satisfied:
\begin{itemize}
\item[(a)] $F_1, \ldots, F_n$ are continuous on all of $X$.
\item[(b)] There exists an open dense subset $U$ contained in the
  smooth locus of $X$ such that $F_1, \ldots, F_n$ are differentiable
  on $U$, and the differentials $dF_1, \ldots, dF_n$ are linearly
  independent on $U$. 
\item[(c)] $F_1, \ldots, F_n$ pairwise Poisson-commute on $U$. 
\end{itemize}
We call the map $\mu := (F_1, \ldots, F_n):X \to \r^n$ the {\it moment map} of the integrable system.
\end{Def}

Fix a K\"ahler form $\omega$ on a complex projective variety $X$.
The main goal of the first part of this manuscript is to 
show that one can use a ``good" toric degeneration of $X$
to construct an integrable system
$\{F_1, \ldots, F_n\}$ on $(X, \omega)$ in the sense of Definition
\ref{def-int-system}. 

An important class of completely integrable systems in symplectic
geometry is given by taking components of the moment map of a
  compact symplectic toric manifold. (The analogous objects in the algebraic 
category are projective toric varieties.)
Let $\T = (\c^*)^n$ denote an algebraic torus. Let $T =
(S^1)^n$ denote the corresponding compact torus and 
let $V$ be a finite-dimensional $\T$-module. The action of $\T$ on $V$ induces an action on the projective space $\p(V)$. 
Suppose $H$ is a $T$-invariant Hermitian form on $V$. Then the
corresponding Fubini-Study K\"ahler form $\Omega_H$ on $\p(V)$
is $T$-invariant. In fact, more is true: the $T$-action on $(\p(V),
\Omega_H)$ is Hamiltonian, with moment map $\mu: \p(V) \to \Lie(T)^*$
given by 
\begin{equation} \label{equ-moment-projective-space}
\langle \mu(x), \xi \rangle = \frac{i}{2} \frac{H(\xi \cdot \tilde{x}, \tilde{x})}{H(\tilde{x}, \tilde{x})}
\end{equation} 
where $\xi \in \Lie(T)$ and 
$\tilde{x} \in V$ represents the point $x \in \p(V)$.  

Let $\A = \{ \alpha_0, \ldots, \alpha_r\} \subset \z^n$ be a finite set of integral points. Let $\T$ act linearly on
$V = \c^{r+1}$ via weights $\alpha_0, \ldots, \alpha_r$, that is:
$$ x \cdot (z_0, \cdots, z_r) = (x^{\alpha_0} z_0, \cdots, x^{\alpha_r} z_r),$$
where $x \in \T$ and we use shorthand notation $x^\alpha = x_1^{a_1} \cdots x_n^{a_n}$ for
$\alpha = (a_1, \ldots, a_n)$, $x = (x_1, \ldots, x_n)$.
We assume that $\A$ is large enough so that the induced action of $\T$ on $\p(V)$ is faithful. Then the orbit 
of a generic point, e.g. $(1: \cdots : 1)$ is isomorphic to $\T$ itself.
Let $X_\A$ denote the closure of the orbit of $(1: \cdots :1)$. Then $X$ is a (not necessarily smooth or normal) projective toric variety.
Take the Fubini-Study K\"ahler form $\Omega$ on $\p(V)$ with respect to the standard Hermitian product on $V = \c^{r+1}$. It is invariant under the action of the real torus $T = (S^1)^n$. Equip $X_\A$ (more precisely the smooth locus of $X_\A$) with the K\"ahler structure induced from $\p(V)$. 
The moment map $\mu$ for the projective space $\p(V)$ restricts to a moment map $\mu_\A: X_\A \to
\Lie(T)^* \cong \r^n$ for $X_\A$. The image of $\mu_\A$ is the convex hull 
$\Delta = \conv(\A)$.  
Let $\{H_1, \ldots, H_n\}$ denote the components of $\mu_\A$ 
with respect to some choice of a basis for $\Lie(T)$. Since $X_\A$ is a 
toric variety, i.e. has a dense open orbit isomorphic to $\T$ itself,
we immediately conclude: 
\begin{itemize}
\item The collection $\{H_1, \ldots, H_n\}$ is an integrable system on 
$X_\A$, in the sense of Definition \ref{def-int-system}.

\item The image of the moment map $\mu_\A = (H_1, \ldots, H_n)$ is the rational polytope $\Delta= \conv(\A)$.

\item By construction, the Hamiltonians $H_1, \ldots, H_n$ generate a
  (Hamiltonian) torus action on $X_\A$.
\end{itemize}

\begin{Rem} \label{rem-toric-int-system-multiple-FS}
{If we replace the K\"ahler form $\Omega$ by $c\Omega$ for some non-zero real number $c$, then the moment map and the moment image for the $T$-action are replaced with $c\mu$ and $c\Delta$ respectively. In particular, the moment image may no longer be an integral polytope.}
{This situation arises later in Part \ref{part-2} where we deal with toric degenerations to projective toric varieties associated to polytopes with rational vertices. 
In this case the K\"ahler form on the toric variety coming from the polytope is the restriction of a constant multiple of the Fubini-Study form.}
\end{Rem}

\section{Gradient-Hamiltonian flow} \label{sec-grad-Hamiltonian}
Let $\X$ be an algebraic variety equipped with a K\"ahler structure $\tilde{\omega}$ on its smooth locus $\X_{smooth}$. 
Let $\pi: \X \to \c$ be a morphism (i.e. an algebraic map). We let $X_t$ denote the fiber $\pi^{-1}(t)$ and $\omega_t$ denote the restriction 
of $\tilde{\omega}$ to the smooth points of $X_t \cap \X_{smooth}$. 
Following Ruan \cite{Ruan}, we now define the
\textbf{gradient-Hamiltonian vector field} corresponding to $\pi$ on
$\X_{smooth}$ as follows. Let $\nabla(\Re(\pi))$ denote the gradient
vector field on $\X_{smooth}$ associated to the real part $\Re(\pi)$, with respect to
the K\"ahler metric. Since $\tilde{\omega}$ is K\"ahler and $\pi$ is holomorphic,
the Cauchy-Riemann equations imply that $\nabla(\Re(\pi))$ is related
to the Hamiltonian vector field 
$\xi_{\Im(\pi)}$ of the imaginary part $\Im(\pi)$ with respect to the
K\"ahler (symplectic) form $\tilde{\omega}$ by  
\begin{equation}\label{gradient Re pi and Hamiltonian Im pi}
\nabla(\Re(\pi)) = - \xi_{\Im(\pi)}.
\end{equation}
Let $Z \subset \X$ denote the closed set which is the union of the singular locus of $\X$ 
and the critical set of $\Re(\pi)$, i.e. the set on which $\nabla(\Re(\pi)) = 0$.  
The {\it gradient-Hamiltonian vector field} $V_\pi$, which is defined only on the open set $\X \setminus Z$,
is by definition 
\begin{equation}\label{def-grad-Hamiltonian}
V_\pi := - \frac{\nabla(\Re(\pi))}{\|\nabla(\Re(\pi))\|^2}.  
\end{equation}
Where defined, $V_\pi$ is smooth.
{Moreover, from~\eqref{def-grad-Hamiltonian} it follows that}
\begin{equation}\label{equ-V-normalization}
V_\pi(\Re(\pi)) = -1.
\end{equation}
For $t \in \r_{\geq 0}$ let $\phi_t$ denote the time-$t$ flow
corresponding to $V_\pi$. 
Since $V_\pi$ may not be complete, $\phi_t$ for a given $t$ is not
necessarily defined on all of $\X \setminus Z$. We record the
following facts. 

\begin{Prop} 
\label{prop-grad-Hamiltonian}
Let the notation be as introduced above.
\begin{itemize}
\item[(a)] Suppose $s,t \in \r$ with $s \geq t > 0$. Where defined, the flow $\phi_t$ takes $X_s \cap (\X \setminus Z)$ to
$X_{s-t}$. In particular, where defined, $\phi_t$ takes a point $x \in
X_t$ to a point in the fiber $X_0$. Moreover, if the $\X$ is smooth
away from $X_0$ and there are no critical points of $\pi$ in
$\pi^{-1}([s, s-t])$, then $\phi_t$ is defined everywhere on $X_s$. 
\item[(b)] Where defined, the flow $\phi_t$ preserves the symplectic
  structures, i.e., if $x \in X_z \cap (\X \setminus Z)$ is a point
  where $\phi_t(x)$ is defined, then
  $\phi_t^*(\omega_{z-t})_{\phi_t(x)} = (\omega_z)_x$. 
\end{itemize}
\end{Prop}

Recall that a smooth point $x \in \X$ is a critical point of
$\pi$ if $d\pi(x)=0$. A scalar $z \in \c$ is a critical value if $X_z$ 
contains at least one critical point. Note that by the Cauchy-Riemann
relations $d\pi = 0$ if and only if $d(\Re(\pi)) = 0$, which in turn
holds if and only if $\nabla(\Re(\pi)) = 0$. 
The next lemma is immediate from generic smoothness \cite[Corollary
III.10.7]{Hartshorne}.

\begin{Lem} \label{lem-cri-value}
Let $\pi: \X \to \c$ be a morphism. 
Then the set of critical values of $\pi: \X_{smooth} \to \c$ is a finite set. In particular there exists $\epsilon > 0$ such that 
there is no critical value in the interval $(0, \epsilon]$. 
\end{Lem}

\begin{Rem} \label{rem-toric-deg-cri-pt}
By definition the families we consider in Section \ref{sec-int-system} (namely toric degenerations) are trivial over $\c^*$. Moreover we assume that 
the generic fiber is smooth. In this case the only possible critical value is $0$ and Lemma \ref{lem-cri-value} is not
needed. 
\end{Rem}

Next we recall the following fundamental theorem on the solutions of ordinary differential equations
(see for example \cite[Theorem 1.7.1]{Hu-Li}):
\begin{Th} \label{th-ODE-continuation}
Let $V$ be a continuously differentiable vector field defined in an open subset $U$ of 
a Euclidean space. Then for any $x_0 \in U$, the flow $\phi_t$ of $V$ is defined at $x_0$ for 
$0 \leq t < b$ where we have the following possibilities: 
\begin{enumerate} 
\item $b = \infty$ or 
\item $b < \infty$ and 
$\phi_t(x_0)$ is unbounded as $t \to b^-$, or 
\item $b < \infty$ and 
$\phi_t(x_0)$ approaches the boundary of $U$ as $t \to b^-$.
\end{enumerate}
\end{Th}

Let $U_0 \subset X_0$ denote the open set $X_0 \setminus (X_0 \cap Z)$. Then the gradient-Hamiltonian vector field 
is defined at all $x_0 \in U_0$. {Note that, a priori,
  the set $U_0$ may be empty; however, in} the next section (Section \ref{sec-flatness-cri-pts}) we prove
that {under some assumptions on the family $\X$,
the set $U_0$ is in fact 
dense in $X_0$.}

\begin{Lem} \label{lem-flow-grad-Hamiltonian}
Let the notation be as above. Assume that the family $\X$ is smooth away from the fiber $X_0$.
Further assume that the map $\pi: \X \to \c$ is proper with respect to the classical topologies on $\X$ and $\c$. 
Let $\epsilon>0$ be such that there are no critical values in
$(0,\epsilon] \subset \c$. Then the flow $\phi_{-\epsilon}$ is defined for all $x_0 \in U_0$.
In particular, if there are no critical values in $(0, +\infty)$ then $\phi_{-t}$ is defined for all $x \in U_0$ and all $t>0$.
\end{Lem}

\begin{proof}
{Let $x_0 \in U_0$. From the theory of ODEs we know
  there exists $b>0$ such that the flow $\phi_{-t}(x_0)$ is defined
  for $t \in [0,b)$. If $b \geq \epsilon$ there is nothing to prove,
  so suppose $b<\epsilon$. From
  Proposition~\ref{prop-grad-Hamiltonian} (a) we also know that
  $\phi_{-t}(x_0) \in X_t$ for each such $t$. Thus the flow defines a
  continuous map from $[0,b)$ to $\pi^{-1}([0,b])$ by $t \mapsto
  \phi_{-t}(x_0)$. Since $\pi$ is proper, the inverse image
  $\pi^{-1}([0,b])$ is a compact set, so there exists a point $x_b \in
  X_b$ such that $\phi_{-t}(x_0) \to x_b$ as $t \to b^-$. By
  assumption on the family $\X$, the fiber $X_b$ lies in the smooth
  locus of $\X$. Moreover by the assumption on $\epsilon$
  we also know that $X_b$ contains no
  critical points of $\pi$. Thus the limit point $x_b$ is contained in
  $\X \setminus Z$ and so by Theorem~\ref{th-ODE-continuation} we
  conclude the flow is defined at $t=b$ and takes the value
  $\phi_{-b}(x_0) = x_b$. Again by the theory of ODEs this implies the
  flow is defined not just on $[0,b)$ but on $[0,b')$ for some
  $b'>b$. Repeating the above argument with $b'$ in place of $b$ as
  necessary shows that the flow is in fact defined also at
  $t=\epsilon$, as desired. }
\end{proof}

\section{Critical points of the family} \label{sec-flatness-cri-pts}
In this section we show that under mild assumptions on the family
$\X$, the set of critical points of $\pi$ intersected with $X_0$ lies in the singular locus of $X_0$. It follows that
the set $U_0 \subseteq X_0$ discussed in the previous section is open and dense in $X_0$ (Corollary \ref{cor-Uzero-open-dense}).

By a flat family of varieties over $\c$ we mean a flat morphism $\pi:
\X \to \c$ on a quasiprojective variety $\X$ such that the fibers are reduced as schemes. Since $\pi$ is
a flat morphism all the fibers have the same dimension $n = \dim(\X) -
1$. Before stating the next result we note that without loss of
generality we can always assume that the family $\X$ is embedded in $\p^N
\times \c$ for some projective space $\p^N$ and that the map $\pi$ is the
restriction of the projection $\p^N \times \c \to \c$ on the second
factor. To see this, let $i: \X \hookrightarrow \p^N$ be an embedding
of the quasi-projective variety $\X$ in a projective space $\p^N$.
Then the map $\tilde{i}: \X \to \p^N \times \c$ given by $\tilde{i}(x)
= (i(x), \pi(x))$ is also an embedding and $\pi = p_2 \circ \tilde{i}$
where $p_2$ is the projection on the second factor, as claimed.

\begin{Prop} \label{prop-sing-family} 
Let $\pi: \X \to \c$ be a flat family of varieties with $\dim(\X) = n+1$. Consider a fiber $X_z = \pi^{-1}(z)$.
If $p \in X_z$ is a smooth point of $X_z$ then it is a smooth point of $\X$ as well.  
\end{Prop}
\begin{proof}
A point $p$ is a smooth point of $X_z$ if and only if $\dim(T_p X_z) = \dim(X_z)$ which in turn is equal to $\dim(\X) - 1$ because of the 
flatness of the family. Since $X_z$ is a Cartier divisor in $\X$ we know $\dim(T_p \X) - 1 \leq \dim(T_p X_z)$.
So $\dim(T_p \X) \leq \dim(\X)$, i.e. $p$ is a smooth point of $\X$.

%Take a point $p \in X_z$ such that $X_z$ is smooth at $p$.
%We wish to show that the family $\X$ is smooth at $p$. Let $\m_{{X_z}, p}$ and $\m_{\X, p}$ 
%denote the maximal ideals of $p$ in the local rings $\mathcal{O}_{{X_z}, p}$ and $\mathcal{O}_{\X, p}$ respectively.
%Since $p$ is a smooth point of $X_z$ we can find a set of generators $g_1, \ldots, g_n$ for $\m_{{X_z}, p}$. 
%Consider $p$ as a point of $\X \subset \p^N \times \c$ and assume 
%that $p$ lies in the affine chart $\c^N \times \c = \c^{N+1}$. Put $\U = \X \cap \c^{N+1}$.
%Let $F_1, \ldots, F_s$ generate the ideal of $\U \subset \c^{N+1}$. Let $t$ denote the last coordinate function, i.e. the coordinate on $\c$.
%By assumption $\X$ is a family of varieties, in other words the fibers of $\pi: \X \to \c^1$ are reduced. It means that the ideal of $X_z \cap \c^{N+1}$ is 
%generated by $F_1, \ldots, F_s$ and $t$. Finally let $G_1, \ldots, G_n$ be rational functions on $\c^{N+1}$ such that $(G_i)_{|X_z} = g_i$ for each $i$.
%It follows that the maximal ideal $\m_{\c^{N+1}, p}$ of the point $p$ in the local ring $\mathcal{O}_{\c^{N+1}, p}$ is generated by 
%$F_1, \ldots, F_s, t, G_1, \ldots, G_n$. Hence the maximal ideal of $\m_{\X, p}$ is generated by the images of $G_1, \ldots, G_n$ and $t$. Since $\dim(\X) = n+1$ 
%this implies that $p$ is a smooth point of $\X$ as required.
\end{proof}

We now apply this to our setting in
Section~\ref{sec-grad-Hamiltonian}.
Recall that $Z \subset \X$ denotes the closed set which is the union of the 
singular locus of $\X$ and the 
critical set of $\pi$, and $U_0$ denotes the (Zariski-open)
subset $X_0 \setminus (X_0 \cap Z)$.

We state the following simple fact whose proof is immediate from the definitions.
\begin{Lem} \label{lem-cri-pt-singular}
Let $\X$ be an algebraic variety and $\pi: \X \to \c$ a morphism. Suppose a fiber $X_z$ is reduced as a scheme. Take a point $p \in X_z$ such that $\X$ is smooth at $p$. If $p$ is 
a critical point of $\pi$, i.e. $d\pi(p) = 0$, then $p$ is a singular point of $X_z$.
\end{Lem}

The following is a simple consequence of Proposition~\ref{prop-sing-family} and the above lemma applied to
the fiber $X_0$. 

\begin{Cor} \label{cor-Uzero-open-dense}
Let $\pi: \X \to \c$ be a flat 
family of varieties. Consider the fiber $X_0$. Then $\pi$ does not have any critical points on the smooth locus of $X_0$.
It follows that $U_0$ is precisely the smooth locus of $X_0$ and in particular $U_0$ is dense in $X_0$. 
\end{Cor}
\begin{proof}
Take $p \in X_0$ which is a smooth point for $X_0$. By Proposition \ref{prop-sing-family}, $p$ is a smooth point of $\X$ as well. Now $p$ cannot be a critical point of $\pi$
because otherwise by Lemma \ref{lem-cri-pt-singular} it should be a singular point of $X_0$ which is a contradiction. This proves the claim.
\end{proof}

For the next corollary we need to equip the family $\X$ with a K\"ahler structure. 
Suppose $\Omega$ is a
  Fubini-Study K\"ahler form on $\p^N$ and suppose $\p^N \times \c$ is
  equipped with the product K\"ahler structure $\Omega \times \left(
    \frac{i}{2} dz \wedge d\bar{z}\right)$ where the second factor is
  the standard K\"ahler structure on $\c$ (here $z$ is the coordinate
  on $\c$). We assume that (the smooth locus of) $\X$ is equipped with
  the K\"ahler structure obtained by restricting this product
  structure to $\X$. 

\begin{Cor}\label{cor-flow-grad-Hamiltonian}
With notation and assumptions as in the discussion above and 
Corollary~\ref{cor-Uzero-open-dense}, let $\epsilon>0$ be as in Lemma~\ref{lem-flow-grad-Hamiltonian} i.e. there are no critical values of $\pi$ in $(0, \epsilon]$. 
Then the image {$U_\epsilon := \phi_{-\epsilon}(U_0)$} of $U_0$ under the flow $\phi_{-\epsilon}$
is an open dense subset of $X_\epsilon$.
\end{Cor}

\begin{proof}
The flow $\phi_{-\epsilon}$ is a local diffeomorphism. 
Since $\phi_{-\epsilon}$ sends $U_0$ to $X_\epsilon$, the image 
$U_\epsilon := \phi_{-\epsilon}(U_0)$ is 
an open subset of $X_\epsilon$. It remains to show that $U_\epsilon$ is dense in $X_\epsilon$. 
To this end, recall that by Proposition \ref{prop-grad-Hamiltonian} 
we know $\phi_{-\epsilon}^*(\omega_\epsilon) = \omega_0$. Thus 
\begin{equation} \label{equ-integral-U-epsilon}
\int_{X_0} \omega_0^n = \int_{U_0} \omega_0^n = \int_{U_0} (\phi_{-\epsilon}^*(\omega_\epsilon))^n 
= \int_{U_\epsilon} \omega_\epsilon^n
\end{equation}
where the first equality uses Corollary~\ref{cor-Uzero-open-dense}. 
By assumption, $\omega_0$ and $\omega_\epsilon$ are restrictions of 
the K\"ahler form $\Omega$ on $\p^N$. Furthermore, by
\cite[Section 5C]{Mumford}, the symplectic volumes $\int_{X_0} \omega_0^n$ and 
$\int_{X_\epsilon} \omega_\epsilon^n$ are equal (up to a normalization factor, which is in fact $n!$) 
to the degrees of $X_0$ and $X_\epsilon$ regarded as 
subvarieties of $\p^N$. 
Since the family $\X$ is flat, 
the degrees of $X_0$ and $X_\epsilon$ are equal (see
\cite[Theorem III.9.9]{Hartshorne}) and thus 
\begin{equation} \label{equ-integral-X-epsilon}
\int_{X_0} \omega_0^n = \int_{X_\epsilon} \omega_\epsilon^n.
\end{equation}
Equations \eqref{equ-integral-U-epsilon} and
\eqref{equ-integral-X-epsilon} together imply that $\int_{U_\epsilon}
\omega_\epsilon^n = \int_{X_\epsilon} \omega_\epsilon^n$. Thus
$X_\epsilon \setminus U_\epsilon$ has empty interior and $U_\epsilon$
is dense in $X_\epsilon$ as desired.
\end{proof}

\section{Continuity of the gradient-Hamiltonian flow} \label{sec-contin}

{The main result of this (technical) section is
  Theorem~\ref{th-flow-continuous}, the proof of which depends on a
  subtle generalization of the famous ``\L{}ojasiewicz gradient
  inequality''. The point of the theorem is to prove 
that for a small enough parameter
$\epsilon > 0$, the flow $\phi_{\epsilon}$ of the gradient-Hamiltonian
vector field on $\X$, which is a priori
not defined at all points of $X_\epsilon$, can be extended
to a continuous function on all of $X_\epsilon$. (The continuity is with respect to the classical topology.) }
We refer the reader to the paper of Lerman \cite{Lerman} for background, where we learned about gradient flows in the context of symplectic geometry. 

{In what follows we continue to place the 
hypotheses on the family $\X$ and its K\"ahler structure from
Sections~\ref{sec-grad-Hamiltonian}
and~\ref{sec-flatness-cri-pts}. Specifically, we assume that}
$\pi: \X \to \c$ is a flat family of irreducible varieties 
where  $\X \subset \p^N \times \c$ as an algebraic subvariety and $\pi$ is the restriction of the projection to the second factor. 
We also assume that $\X$ is smooth away from the singular locus of $X_0$. 
We fix a K\"ahler form $\tilde{\omega}$ on $\X$ which is the
restriction of a product K\"ahler form on $\p^N \times \c$ of the form
$\Omega \times \left(\frac{i}{2} dz \wedge d\bar{z}\right)$ where
$\Omega$ is a Fubini-Study K\"ahler structure on $\p^N$. Recall that $\phi_t$ is the flow associated to the 
gradient-Hamiltonian vector field defined in Section~\ref{sec-grad-Hamiltonian} and $U_0$ is the open set $X_0 \setminus (X_0 \cap Z)$ where $Z$ is the closed set consisting of the union of the singular locus of $\X$ and the critical set of $\pi$. Also recall that by Lemma~\ref{lem-flow-grad-Hamiltonian} and Corollary~\ref{cor-flow-grad-Hamiltonian} we know there exists an $\epsilon>0$ such that $U_\epsilon := \phi_{-\epsilon}(U_0)$ is open and dense in $X_\epsilon$.

\begin{Th} \label{th-flow-continuous}
With the notation and assumptions as above, let $\epsilon > 0$ satisfy the condition of
Lemma~\ref{lem-flow-grad-Hamiltonian}, i.e. there are no critical values in the interval $(0, \epsilon]$. 
Then the flow $\phi_\epsilon$ of the
gradient-Hamiltonian vector field associated to $f := \Re(\pi)$, a priori defined only on {the (open and dense) subset}
$U_\epsilon \subseteq X_\epsilon$, extends 
to a continuous {(with respect to the classical topology)} function defined on all of $X_\epsilon$, with image in $X_0$.
\end{Th}

\begin{Rem} \label{rem-flow-continuous}
In general, the flow with respect to the gradient vector field of a
smooth function $f$ on a Riemannian manifold $X$ does not give rise to
a continuous function between the level sets of $f$. 
The reason that we can continuously extend our flow $\phi_\epsilon$
is that we work in an algebraic (and hence real analytic) setting. 
\end{Rem}

The proof of Theorem \ref{th-flow-continuous} is based on the famous
``\L{}ojasiewicz gradient inequality'' for analytic functions
(cf. \cite{Loj} {and}
\cite[pages 763 and 765]{KMP}):

\begin{Th}\label{th-Lojasiewicz}
Let $f$ be a real-valued analytic function defined on some open subset $W \subset \r^m$. 
Then for any $x \in W$ there exists an open neighborhood $U_x$ of $x$ and constants $c_x > 0$ and $0 < \alpha_x < 1$ such that for all $y \in U_x$:
\begin{equation} \label{equ-Lojasiewicz}
|| \nabla f(y) || \geq c_x |f(y) - f(x)|^{\alpha_x}.
\end{equation}
\end{Th}

In fact, since $X_0$ may not be smooth, we will need the following generalization of Theorem~\ref{th-Lojasiewicz} due to 
Kurdyka and Parusinski \cite[Proposition 1]{Kurdyka}. 
 
\begin{Th}\label{th-general-Lojasiewicz}
Let {$Y$} be a (possibly singular) algebraic subset of $\r^m$. 
Let $f:\r^m \to \r$ be a semi-algebraic function.
Then for any {$x \in Y$} (not necessarily a smooth point) there is an open neighborhood {$U_x \subset Y$} (in the {classical} 
topology) and constants $c_x > 0$ and $0 < \alpha_x < 1$ such that for any smooth point $y \in U_x$ we have:
\begin{equation} \label{equ-Lojasiewicz-general}
||\nabla f(y) || \geq c_x |f(y) - f(x)|^{\alpha_x},
\end{equation}
where $\nabla f$ denotes the gradient of $f_{|Y}$ with respect to the induced metric on the smooth locus of 
{$Y$}.  
\end{Th}

\begin{Rem}
Since any Riemannian metric on a relatively compact subset of $\r^n$ (i.e. a subset with 
compact closure) induces a norm which is equivalent\footnote{We say that two norms $\lvert \cdot \rvert_1$ and $\lvert \cdot \rvert_2$ are 
equivalent if there exist constants $c, C>0$ such that $c \lvert \cdot \rvert_1 < \lvert \cdot \rvert_2 < C \lvert \cdot \rvert_1$.} 
to the standard Euclidean norm, the inequalities \eqref{equ-Lojasiewicz} and
\eqref{equ-Lojasiewicz-general} hold for an arbitrary Riemannian metric on $\r^m$ with the same exponent $\alpha_x$
and possibly different constant $c_x$.
\end{Rem}

\begin{Rem}
In \cite{Kurdyka} the more general case of subanalytic functions and
sets is addressed. For our purposes,
the algebraic/semi-algebraic case stated above suffices. 
\end{Rem}

\begin{proof}[Proof of Theorem \ref{th-flow-continuous}]
Let $f := \Re(\pi)$. For $x \in X_0$, let $U_x$, $c_x > 0$ and $0 < \alpha_x < 1$ be as in Theorem \ref{th-general-Lojasiewicz}
{(applied to the family $\X$)}. 
Since $X_0$ is compact, it can be covered with finitely many such open sets $U_{x_1}, \ldots, U_{x_s}$. Let
$$U := \left( \bigcup_{i=1}^s U_{x_i} \right) \cap \{y \in \X \mid |f(y)| < \epsilon\} {\subseteq \X}$$
$$c := \min\{c_{x_1}, \ldots, c_{x_s}\},$$
$$\alpha := \max\{\alpha_1, \ldots, \alpha_s\}.$$
Clearly $X_0 \subset U$. {Moreover, from~\eqref{equ-Lojasiewicz-general} 
it follows that 
for any point $y \in U$ we have:} 
\begin{equation} \label{eq-Loj-ineq}
|| \nabla f(y) || \geq c |f(y)|^{\alpha}. 
\end{equation}
Next we claim that there exists
$\rho > 0$ such that $X_t \subset U$ for all $0 < t < \rho$. To see
this, suppose for a contradiction that 
there exists a sequence $(x_i)_{i \in \n} \in \X$ such that for all $i$ we
have $x_i \notin U$, $0 < \pi(x_i) < 1$, and 
$\lim_{i \to \infty} \pi(x_i) = 0$. {Since $\pi$ is proper, the set $\pi^{-1}([0,1]) \subseteq \X$ is compact, from which}
it follows that the sequence $(x_i)$ has a limit point $x \in \X$. 
Since $x_i \notin U$ for all $i$ and $U$ is open, we know $x \notin
U$. On the other hand, 
by continuity $\pi(x) = 0$ and hence $x \in X_0 \subset U$,
contradiction. 

{Now let $\rho>0$ be as above. From \eqref{eq-Loj-ineq} we conclude that for any $0 < t < \rho$ and $y \in X_t$ we have:}
\begin{equation}\label{eq-adjusted-Loj-ineq}
|| \nabla f(y) ||^{-1} \leq (1/c) |f(y)|^{-\alpha}.
\end{equation}
{Note that $f(y) = \Re(t) = t$ is nonzero and hence division by $f(y)$ is allowed.}

Also for any $\epsilon-\rho < t < \epsilon$ and
for any $y \in X_\epsilon$ we have {that $\phi_t(y)$ is
defined (by an argument similar to the proof of
Lemma~\ref{lem-flow-grad-Hamiltonian}) and that} $$\phi_t(y) \in X_{\epsilon-t} \subset U.$$
Let $x \in X_\epsilon$. Then for any $t_0, t_1$ with $\epsilon-\rho < t_0 < t_1 <
\epsilon$, we have 
\begin{eqnarray*}
\left\lVert \phi_{t_1}(x) - \phi_{t_0}(x)  \right\rVert &=& \left\lVert \int_{t_0}^{t_1} \frac{d}{dt} \phi_t(x) dt \right\rVert \cr
&=& \left\lVert \int_{t_0}^{t_1} V( \phi_t(x)) dt \right\rVert \cr
&\leq& \int_{t_0}^{t_1} \left\lVert V(\phi_t(x)) \right\rVert dt \cr 
\end{eqnarray*}
In addition, by~\eqref{eq-adjusted-Loj-ineq} we have 
$$|| \phi_{t_1}(x) - \phi_{t_0}(x) || \leq \int_{t_0}^{t_1} (1/c)|f(\phi_t(x))|^{-\alpha} dt$$
and since $x \in X_\epsilon$, by Proposition~\ref{prop-grad-Hamiltonian} we have $$f(\phi_t(x)) = \epsilon-t.$$ 
Putting this together we obtain 
\begin{eqnarray*}
\left\lVert \phi_{t_1}(x) - \phi_{t_0}(x) \right\rVert &\leq& \frac{1}{c} \int_{t_0}^{t_1} (\epsilon-t)^{-\alpha} dt \cr
&\leq& \frac{1}{c(1-\alpha)}((\epsilon-t_1)^{1-\alpha} - (\epsilon-t_0)^{1-\alpha}). \cr
\end{eqnarray*}
Since $1-\alpha > 0$ by assumption, we know $\lim_{t \to \epsilon}
(\epsilon-t)^{1-\alpha} = 0$. Therefore for any
$\varepsilon'>0$, 
there exists $\delta > 0$ such that for any $x \in X_1$ and any $t_0,
t_1$ with $1-\delta < t_0 < t_1 < 1$, we have
$$|| \phi_{t_1}(x) - \phi_{t_0}(x) || \leq \varepsilon'.$$ 
It follows that for any $x \in X_\epsilon$, $\lim_{t \to \epsilon} \phi_{t}(x)$ exists. 
Let $\phi(x)$ denote this limit. By continuity,
if the original flow $\phi_\epsilon$ was already defined at $x$, then
$\phi_\epsilon(x) = \phi(x)$. Thus the map $\phi$ is an extension of $\phi_\epsilon$
to all of $X_\epsilon$, and by construction takes $X_\epsilon$ to $X_0$. 

It remains to 
show that $\phi: X_\epsilon \to X_0$ is continuous. 
Let $x \in X_\epsilon$. We need to show that for any $\varepsilon'' > 0$ there exists $\delta > 0$ such that if $y \in X_\epsilon$ and 
$||y - x|| < \delta$ then $||\phi(y) - \phi(x)|| < \varepsilon''$. 
By the above, we can choose $t \in \r$ with $0 < t < \epsilon$ such that for any $y \in X_\epsilon$ 
we have $$|| \phi_t(y) - \phi(y) || < \varepsilon''/3.$$
Since the map $y \mapsto \phi_t(y)$ is continuous, we can find $\delta > 0$ such that 
for any $y \in X_\epsilon$, if $|| x - y || < \delta$ then {$|| \phi_t(x) - \phi_t(y) || < \varepsilon'' / 3$.} 
Now, if $|| x - y || < \delta$, we have:
\begin{eqnarray*}
||\phi(x) - \phi(y)|| &\leq& || \phi(x) - \phi_t(x)|| + ||\phi_t(x) - \phi_t(y)|| + ||\phi_t(y) - \phi(y)|| \cr
&\leq& \varepsilon''/3 + \varepsilon''/3 + \varepsilon''/3 \cr
&\leq& \varepsilon'' \cr
\end{eqnarray*}
as desired. 
\end{proof}

\section{Construction of an integrable system from a toric degeneration} \label{sec-int-system}

With these preliminaries in place, we can now prove the main result of
{Part 1 of this manuscript, which constructs an integrable system 
in the sense of Definition~\ref{def-int-system} from a toric degeneration satisfying certain properties (described below)}. 

{We first recall the definition of a toric
  degeneration.} Let $X$ be an $n$-dimensional projective variety. Let $\pi: \X \to \c$ be a family of varieties. 
\begin{Def} \label{def-toric-degen}
We call $\pi: \X \to \c$ a {\bf toric degeneration of $X$} if the following hold:
\begin{itemize}
\item[(1)] $\pi: \X \to \c$ is a flat family of irreducible varieties. In particular, each fiber $X_t := \pi^{-1}(t)$ is reduced regarded as a scheme. 
\item[(2)] The family $\X$ is trivial over $\c^*$, i.e. there is a fiber-preserving isomorphism {of varieties}
$\rho: X \times \c^* \to \X \setminus X_0$. In particular, for any $t
\in \c^*$, {the restriction $\rho_t$ of $\rho$ to $X
  \times \{t\}$ yields an isomorphism of the varieties $X \cong X
  \times \{t\}$ and the 
  fiber} $X_t$. 
\item[(3)] The fiber $X_0$ is a toric variety {with respect to an action of $\T = (\c^*)^n$}. 
\end{itemize} 
\end{Def}
For the rest of our discussion we fix once and for
all the isomorphism $\rho: X
\times \c^* \to \X \setminus X_0$ in item (2) above. Consider the
diagonal action of $\c^*$ on $X \times \c^*$, where $\c^*$ acts trivially on the 
first factor and standardly
on the second factor. The identification $\rho$ then
induces a corresponding 
$\c^*$-action on $\X \setminus X_0$ which preserves the fibers of $\pi$. 

Next we state some hypotheses, listed in items (a)-(d) below, which we place on our toric degenerations. These hypotheses are natural in the 
sense that many of the known examples satisfy them (cf. Section~\ref{sec-examples} and Remark~\ref{remark-natural} below). 
For the following, suppose that the projective variety $X$ is equipped with a K\"ahler form $\omega$;
the integrable system constructed in Theorem~\ref{th-A} is with respect to this $\omega$.

Let $\pi: \X \to \c$ be a toric degeneration of $X$. We assume that:
\begin{itemize}
\item[(a)] The family $\X$ is smooth away from the singular locus of
the zero fiber $X_0$. 
\item[(b)] The family $\X$ is embedded in $\p^N \times \c$ as an algebraic subvariety, for some
projective space $\p^N$, such that: 
\begin{itemize} 
\item the map 
$\pi$ is the restriction {to $\X$ of the usual 
  projection of $\p^N \times \c$ to its second factor}, and 
\item the action of $\T$ on $X_0$ extends to a linear
    action on $\p^N$. 
\end{itemize}
\end{itemize}

{We also place some additional hypotheses on the
  K\"ahler form $\omega$ on $X$ as follows.} 
{Suppose $\Omega$ is a constant multiple of a Fubini-Study K\"ahler
form on $\p^N$. Let $z$ denote the complex variable on $\c$ so that
$\omega_{std} = \frac{i}{2} dz \wedge d\bar{z}$ is the standard
K\"ahler form on $\c$. We equip $\p^N \times \c$ with the product
K\"ahler structure $\Omega \times \omega_{std}$. }
{Assuming that $\X$ satisfies the conditions (a) and (b) above, so 
in particular $\X$ is embedded in $\p^N \times \c$, let}
$\tilde{\omega}$ (respectively $\omega_t$) denote the restriction of this product structure to the
smooth locus of the family $\X$ (respectively the fiber $X_t$).  
In this setting we also assume: 
\begin{itemize}
\item[(c)] {The map $\rho_1: X \to X_1$ is an
    isomorphism of K\"ahler manifolds, i.e. $\rho_1^*(\omega_1) =
    \omega$. }
\item[(d)] {Let $T=(S^1)^n$ denote the compact subtorus of $\T$.} The K\"ahler form $\Omega$ on $\p^N$ is $T$-invariant; {in particular the restriction $\omega_0$ to the toric variety $X_0$ is also a $T$-invariant 
K\"ahler form.} 
\end{itemize}
We can now state and prove one of the main results of this
manuscript. 

%{\BL
%\begin{Th}
%Let $\pi: \X \to \c$ be a flat family of irreducible projective varieties with constant smooth fibers equipped with a K\"ahler form $\tilde{\omega}$ 
%coming from embedding in a projective space.  
%Assume that there are no critical values of $\pi$ in the interval $[\epsilon, 0]$. Then there
%is a surjective continuous map $\phi: X_\epsilon \to X_0$ which is a symplectomorphism restricted to a dense open subset $U \subset X_\epsilon$.
%\end{Th}
%}

\begin{Th} \label{th-A} 
Let $X$ be a smooth $n$-dimensional projective variety and let $\omega$ be a K\"ahler structure on $X$.
Suppose that there exists a toric degeneration $\pi: \X \to \c$ of $X$ satisfying
the above conditions (a)-(d). Then:
\begin{itemize}
\item[(1)] There exists a surjective continuous map $\phi: X \to X_0$ which is a 
symplectomorphism restricted to a dense open subset $U \subset X$.
\item[(2)] There exists a completely integrable system $\mu = (F_1, \ldots, F_n)$ on $(X, \omega)$, in the sense of 
Definition \ref{def-int-system}, such that its moment image $\Delta$ coincides with the moment image of $(X_0, \omega_0)$ (which is a polytope). 
\item[(3)] Let $U \subset X$ be the open dense subset of $X$ from (1). Then 
the integrable system generates a torus action on $U$ and the inverse image {$\mu^{-1}(\Delta^\circ)$ of
the interior of $\Delta$ under the moment map
$\mu: X \to \r^n$ of the integrable system}
lies in the open subset $U$.
\end{itemize}
\end{Th}

\begin{Rem}
For part (1) in Theorem \ref{th-A}, 
we do not need to assume that $X_0$ is a toric variety and that the K\"ahler form $\omega_0$ on it is torus invariant.
\end{Rem}

\begin{Rem}\label{remark-natural}
The conditions (a)-(d) above are natural in the sense
that they are satisfied in many of the known examples in the
literature of constructions of toric degenerations. 
Indeed, the main purpose of Part 2 of this manuscript is to show that,
in the general construction of toric degenerations in the context of
the theory of Newton-Okounkov bodies, all the above conditions (a)-(d)
are satisfied. These degenerations come from degenerating polynomials
to their leading terms, often called a Gr\"obner or SAGBI degeneration. 
We also discuss concrete examples in Section~\ref{sec-examples}.
\end{Rem}

\begin{proof}[Proof of Theorem~\ref{th-A}]
Since the family is trivial away from $0$ we know that there are no critical values of $\pi$ on the interval $(0, +\infty)$. 
Taking $\epsilon =1$, Theorem~\ref{th-flow-continuous} tells us that the flow $\phi_{1}: U_1 \to U_0$ extends to 
a continuous function $\phi: X_1 \to X_0$. Hence we have a sequence of 
maps 
$$\xymatrix{
X \ar[r]^{\rho_1} & X_1 \ar[r]^{\phi} & X_0\\
}$$
where $\rho_1$ is a symplectomorphism and $\phi$ is continuous and its 
restriction to an open dense subset $U_1$ is
a symplectomorphism. This proves part (1). Note that $\phi$ is surjective.

Let $\mu_0 = (H_1, \ldots, H_n)$ be the moment map for the $T$-action on the toric variety $X_0$ with respect to the
K\"ahler form $\omega_0$, which is a constant multiple of a Fubini-Study K\"ahler form on the 
projective space $\p^N$ (Section \ref{sec-prelim-part1} and Remark \ref{rem-toric-int-system-multiple-FS}). Let $\Delta$ 
be the moment image, that is the image of $\mu_0$. 
Define a collection of functions $\{F_1, \ldots, F_n\}$ on $X$ by
$$F_i =  H_i \circ \phi \circ \rho_1,  \textup{ for all $1 \leq i \leq n$} $$
We claim that the $\{F_1, \ldots, F_n\}$ form a completely integrable system on $X$ in the sense of Definition \ref{def-int-system}. 
First, by construction the functions $\{H_1, \ldots, H_n\}$ pairwise Poisson-commute with respect to $\omega_0$ and 
their differentials are linearly independent on a dense open set $U_0$
in $X_0$. Since $\phi \circ \rho_1$ is a symplectomorphism from
$\rho_1^{-1}(U_1)$ to $U_0$, 
it follows that the $\{F_1, \ldots, F_n\}$ also pairwise Poisson-commute with respect to
$\omega$ {and their differentials are linearly independent} on the dense open subset $\rho_1^{-1}(U_1)$ of $X$. 
The continuity of the $F_i$ follows from the continuity of 
$\phi \circ \rho_1$ on all of $X$ and the continuity of the $H_i$. Finally, since $X$ is compact, 
the image of the continuous map 
$\phi \circ \rho_1: X \to X_0$  is closed in $X_0$. On the other hand, this image contains the dense 
open set $U_0$ and hence 
$\phi \circ \rho_1$ is surjective. It follows that the image of $(F_1, \ldots, F_n): X_1 \to \r^n$ is the same as 
the image of $(H_1, \ldots, H_n):X_0 \to \r^n$, which is the polytope
$\Delta = \Delta(R)$, as desired. This finishes the proof of (2). To prove (3) note that 
the gradient-Hamiltonian flow is undefined exactly at the singular points or the critical points of $f = \Re(\pi)$. By Corollary \ref{cor-Uzero-open-dense} we know that
there are no critical points of $f$ on the smooth locus of $X_0$. The first assertion in (3) follows by noting that the set of singular points 
of $X_0$ is $T$-invariant. To prove the second assertion let $\Delta^\circ$ be the interior of the polytope $\Delta$. We know the the inverse image 
$\mu_0^{-1}(\Delta^\circ)$ is the open $\T$-orbit (isomorphic to $\T$ itself). Note that under the assumption that there are no critical 
points, the open $\T$-orbit in $X_0$ (hence $\mu_0^{-1}(\Delta^\circ)$) 
is contained in $U_0$. Now since the continuous map 
$\phi: X_1 \to X_0$ is a homeomorphism between $U_1$ and $U_0$, it 
should map $\partial U_1$ to $\partial U_0$. That is, the inverse image of $U_0$ under $\phi$ is $U_1$. This proves the assertion.
\end{proof}

%\begin{Rem} 
%The above proof of Theorem~\ref{th-A} shows, in particular, that the
%continuous map from the generic fiber $X$ to the special fiber $X_0$
%is \emph{onto}. 
%\end{Rem}

\begin{Rem} \label{rem-union-of-toric-var} 
As mentioned in the introduction, the same proof goes through to show that Theorem \ref{th-A}  
still holds if the special fiber $X_0$ is a union of toric varieties and hence possibly 
reducible as a variety. In this case, the moment image is a union of convex polytopes. 
\end{Rem}

\section{Integrable systems and (GIT and symplectic) quotients} \label{sec-GIT-part1}

The goal of this section is to show 
that the construction of an integrable system using a toric degeneration in Sections~\ref{sec-int-system} 
can be made compatible with the presence of a torus action. To accomplish this goal, however, the toric
degeneration $\X$ must be compatible with the torus action in the sense which will be explained below.

We use the setting and notation of Section \ref{sec-int-system}. 
As before let $\T = (\c^*)^n$ denote the $n$-dimensional algebraic
torus with lattice of characters $\z^n$.  Now suppose an algebraic subtorus $\HH \subset \T$ of dimension $m$
acts on the variety $X$. We denote the maximal compact subgroup $T \cap \HH$ of $\HH$ by $H$.

By assumption the action of $\T$ on $X_0$ lifts to an action of $\T$ on $\p^N \times \{0\}$ induced by a linear action of $\T$ on $\c^{N+1}$.
We let $\HH$ act on $\p^N \times \c$ where it acts trivially on $\c$ and acts on $\p^N$ by restricting the action of $\T$ to $\HH$.
{Hence in particular the $\HH$-action on $X_0$ is the restriction of the $\T$-action on $X_0$.}

We assume that the following compatibility conditions hold for the action of $\HH$ and the family $\X \subset \p^N \times \c$.
\begin{itemize}
\item[(i)] The K\"ahler form $\Omega$ on $\p^N$ is invariant under the action of the real torus $H$.

\item[(ii)] The family $\X$ is an $\HH$-invariant subvariety of $\p^N \times \c$ and the action of $\HH$ on $X_1$ coincides with the 
action of $\HH$ on $X$ via the isomorphism $\rho_1: X_1 \to X$. In
other words, the $\HH$-action on $\X$ commutes with the $\c^*$-action
on $\X$. 

\end{itemize}

{Let $\X^{ss} \subset \X$ denote the set of semistable points of $\X$ with respect to the $\HH$-action, 
i.e. $(x,t) \in \X$ for which there exists an $\HH$-invariant polynomial $f$ on $\c^{N+1} \times \c$ and homogeneous on the 
$\c^{N+1}$ factor such that $f(x, t) \neq 0$. 
Recall that $\HH$ acts on $\c^{N+1} \times \c$ by acting linearly on the first factor and acting trivially on the second factor}. 
Let $p: \X^{ss} \to \X' = \X^{ss}/\HH$ be the quotient map. Then since the action of $\HH$ preserves fibers, one verifies that there is a 
map $\pi': \X' \to \c$ such that the diagram 
\begin{equation} \label{equ-comm-diag}
\xymatrix{
\X^{ss} \ar[rr]^{p} \ar[rd]_\pi & & \X' \ar[ld]^{\pi'} \\
& \c &\\
}
\end{equation}
commutes. Moreover, the general fibers of $\X$ and $\X'$ are $X$ and $X' = X // \HH$,
respectively, and the special fibers are $X_0$ and $X'_0 = X_0 // \HH$,
respectively.  
 
Let $\tilde{\mu}_H: \X \to \Lie(H)^*$ denote the moment map of the $H$-action
on the family $\X$ with respect to $\tilde{\omega}$. 
Moreover, we assume that
$\tilde{\mu}_H$ is proper on the set of smooth points on $\X$, $0$ is a regular
value for $\tilde{\mu}_H$, and $H$ acts with finite stabilizer on $\tilde{\mu}_H^{-1}(0)$. 
Then by the Kempf-Ness theorem, $\X'$ can also be identified with the
symplectic quotient $\tilde{\mu}_H^{-1}(0)/H$. 

The K\"ahler structure on the quotient $\X' \cong \tilde{\mu}_H^{-1}(0)/H$ can
be described explicitly as follows. Let $(\tilde{\omega}, \tilde{g})$
denote the K\"ahler (symplectic) form and the corresponding Riemannian
metric on $\X_{smooth}$, and let $(\tilde{\omega}', \tilde{g}')$ denote the
induced K\"ahler structure on $\X'_{smooth}$. The form $\tilde{\omega}'$ is
defined by the relation 
\begin{equation}\label{eq:def quotient symplectic form} 
p^* \tilde{\omega}' = \iota^* \tilde{\omega}
\end{equation}
where $p: \tilde{\mu}_H^{-1}(0) \to \tilde{\mu}_H^{-1}(0)/H \cong \X'$ is the quotient
map and $\iota: \tilde{\mu}_H^{-1}(0) \hookrightarrow\X$ denotes the inclusion. The
quotient Riemannian metric $\tilde{g}'$ is defined by the formula 
\begin{equation}
\label{eq:definition quotient metric}
\tilde{g}_x(v, w) = \tilde{g}'_{p(x)}(dp_x(v), dp_x(w))
\end{equation}
where $v$ and $w$ are required to lie in the orthogonal complement to
the tangent space $T_x(H \cdot x)$ of the $H$-orbit $H \cdot x$
through $x$. 

Before proving the technical proposition leading to our
Theorem~\ref{th-int-system-GIT-part1}, we recall a simple fact about moment
maps. 
\begin{Lem} \label{lem-moment-Hamiltonian-flow}
Let $K$ be a compact Lie group. Suppose $(M, \omega)$ is a Hamiltonian
$K$-space with moment map $\mu_K: M \to \Lie(K)^*$. Let $f$ be a
smooth function on $M$ which is $K$-invariant. Then the Hamiltonian
flow of $f$ preserves $\mu_K$, i.e., $\mu_K$ is constant along the
Hamiltonian vector field of $f$. 
\end{Lem}

Let $\phi_t$ and $\phi'_t$ denote the gradient-Hamiltonian flows on $\X$ and $\X'$ respectively.
Note that an immediate corollary of the above lemma is that the level set
$\tilde{\mu}_H^{-1}(0)$ is invariant under the gradient-Hamiltonian flow
$\phi_t$, since $\pi$ (hence $\Re(\pi)$) is $H$-invariant. Thus it
makes sense to talk about the gradient-Hamiltonian flow restricted to
$\tilde{\mu}_H^{-1}(0)$. 
The following proposition states that, in our setting, the
gradient-Hamiltonian flow is compatible with taking GIT (symplectic)
quotient. 

\begin{Prop} \label{prop-flow-GIT}
The quotient morphism $p: \tilde{\mu}_H^{-1}(0) \to \X' \cong \tilde{\mu}_H^{-1}(0)/H$
commutes with the gradient-Hamiltonian flows.
More precisely, for any $x \in \tilde{\mu}_H^{-1}(0)$ and any $t > 0$ such
that 
both $\phi_t(x)$ and $\phi'_t(p(x))$ are defined, we have
$$p(\phi_t(x)) = \phi'_t(p(x)).$$
\end{Prop}

\begin{proof}
It is enough to show that, whenever defined
we have 
\begin{equation} \label{equ-V-V'}
dp_x(V_x) = V'_{p(x)}
\end{equation}
for $x \in \tilde{\mu}_H^{-1}(0)$. Here 
$V$ and $V'$ denote the gradient-Hamiltonian vector fields of
$\Re(\pi)$ and $\Re(\pi')$, respectively. Recall that
by~\eqref{gradient Re pi and Hamiltonian Im pi} we have 
\begin{equation*}
\begin{split}
V_x = \xi_h(x) / \lVert \xi_h(x) \rVert \\
V'_{x'} = \xi_{h'} (x') / \lVert \xi_{h'}(x') \rVert
\end{split}
\end{equation*}
for $x \in \X, x' \in \X'$, where $h = \Im(\pi)$, $h'=\Im(\pi')$, $\xi_{h}$, $\xi_{h'}$ their Hamiltonian vector fields
respectively, and 
$\lVert \cdot \rVert$ denotes the norm with respect to the metrics $\tilde{g}$ and
$\tilde{g}'$ respectively. 
Let $x \in \tilde{\mu}_H^{-1}(0) \subset \X^{ss}$.  
By~\eqref{eq:def quotient symplectic form} we have $p^*(\tilde{\omega}'_{p(x)})
= \tilde{\omega}_{|{\tilde{\mu}_H^{-1}(0)}}$. Since $h = h' \circ p$ on
$\tilde{\mu}_H^{-1}(0)$ by \eqref{equ-comm-diag} we see that
\begin{equation} \label{equ-xi-h-h'}
dp_x(\xi_h) = \xi_{h'}(p(x)).  
\end{equation}
On the other hand, from the
fact that $\pi$ (and hence $f=\Re(\pi)$) is $H$-invariant, we also
know 
that the gradient $\nabla f(x)$ is orthogonal to the orbit $T_x(H
\cdot x)$ for any $x \in \tilde{\mu}_H^{-1}(0)$. 
Recalling that $V_x$ is also equal to $\nabla
f(x)/\lVert \nabla f(x) \rVert$, applying~\eqref{eq:definition quotient metric}
to $v = w = \xi_h(x)$ we can conclude
\begin{equation} \label{equ-norm-h-h'}
|| \xi_h(x) || = || \xi_{h'} (p(x))||.
\end{equation}
The result follows. 
\end{proof}
{
Let $\mu_H: X \to \Lie(H)^*$ denote the moment map of the $H$-action on $X$ with moment image $\Delta_H$ 
and let $\mu: X \to \r^n$ be the integrable system constructed in Theorem \ref{th-A} with moment image $\Delta$.
Then $\tilde{\mu}_{H, 1}: X_1 \to \Lie(H)^*$ can be identified with $\mu_H$ via the isomorphism $\rho_1: X \to X_1$.  
Let $\mu_{0}$ and $\mu_{H, 0}$ denote the moment maps of $T$-action and $H$-action on the toric variety $X_0$ with moment images $\Delta$ and $\Delta_H$ respectively (Section \ref{sec-prelim-part1}). Then $\mu_{H, 0} = \tilde{\mu}_{H|X_0}$.
Let $i^*: \Lie(T)^* \to \Lie(H)^*$ denote the dual linear map to the inclusion 
$i: \Lie(H) \hookrightarrow \Lie(T)$. Let $X'_0 := \pi^{-1}(0)$ be the fiber over $0$ of the family $\X'$. It can be regarded as the 
GIT quotient of $X_0$ with respect to $\HH$. It is a possibly singular orbifold. By the Kempf-Ness theorem (for orbifolds) we can 
identify $X'_0$ with the symplectic quotient $\mu_{H, 0}^{-1}(0) / H$. Note that $\mu_0$ is $T$-invariant and hence $H$-invariant and thus 
gives a map on $X'_0 = \mu_{H, 0}^{-1}(0) / H$ which we denote by $\mu'_0$. Let $\Delta'$ be 
the intersection of $\Delta$ with the linear subspace $\ker(i^*)$. We have a commutative diagram:
\begin{equation} \label{equ-comm-diag-GIT1}
\xymatrix{
& \Delta_H \ar@{}[r]|{\subset} & \Lie(H)^*\\
X_0 \ar[r]^{\mu_0} \ar[ru]^{\mu_{H, 0}} & \Delta \ar@{}[r]|{\subset} \ar[u]_{i^*} & \Lie(T)^* \\
\mu_{H, 0}^{-1}(0) \ar@{^{(}->}[u]^i \ar[d]_p & \\
X'_0 \ar[r]^{\mu'_0} & \Delta' \ar@{^{(}->}[uu] &\\
} 
\end{equation}
On the other hand,  
there are open dense subsets $U' \subset X'$ and $U_0' \subset X_0'$ and a symplectomorphism $\phi': U' \to U_0'$ given by the 
gradient-Hamiltonian flow. By Proposition \ref{prop-flow-GIT} we know that $\phi' \circ p = p \circ \phi$. Also $\phi: U \to U_0$ is $H$-equivariant and hence
that its continuous extension $\phi: X \to X_0$ is also $H$-equivariant. From these one verifies that $\phi'$ extends to a continuous map
$\phi': X' \to X_0'$ and we have the commutative diagram:
\begin{equation} \label{equ-comm-diag-GIT2}
\xymatrix{
& \Delta_H\\
X \ar[r]^{\phi} \ar[ru]^{\mu_{H}} & X_0 \ar[u]_{i^*}\\
\mu_{H}^{-1}(0) \ar[r]^{\phi} \ar@{^{(}->}[u] \ar[d]_p & \mu_{H, 0}^{-1}(0) \ar@{^{(}->}[u] \ar[d]_p\\
X' \ar[r]^{\phi'} & X'_0\\
} 
\end{equation}
Putting \eqref{equ-comm-diag-GIT1} and \eqref{equ-comm-diag-GIT2} together we obtain the following.
\begin{Th}\label{th-int-system-GIT-part1}
The following diagram is commutative:
\begin{equation} \label{equ-comm-diag-GIT3}
\xymatrix{
& \Delta_H \ar@{}[r]|{\subset} & \Lie(H)^*\\
X \ar[r]^{\mu} \ar[ru]^{\mu_H} & \Delta \ar@{}[r]|{\subset} \ar[u]_{i^*} & \Lie(T)^* \\
\mu_H^{-1}(0) \ar@{^{(}->}[u]^i \ar[d]_p & \\
X' \ar[r]^{\mu'} & \Delta' \ar@{^{(}->}[uu] &\\
}
\end{equation}
\end{Th}

}
\part{Toric degenerations from valuations and Okounkov bodies} \label{part-2}

In Part 2 of this paper, we consider toric degenerations arising from
the theory of Newton-Okounkov bodies. We will show that, for this class
of toric degenerations, all of the technical assumptions we require on
the toric degeneration in
Part 1 (in order to construct integrable systems) are satisfied. 

\section{Preliminaries for Part 2} \label{sec-prelim-part2}

Let $X$ be a projective variety of dimension $n$ over $\c$ equipped
with a very ample line bundle $\lb$.  Let $L
:= H^0(X, \lb)$ denote the space of global sections of $\lb$; it is a
finite dimensional vector space over $\c$. The line bundle $\lb$ gives
rise to the \emph{Kodaira map $\Phi_L$ of $L$}, 
from $X$ to the projective space $\p(L^*)$, 
defined as follows: the image 
$\Phi_L(x)$ of a point $x \in X$ is the point in $\p(L^*)$ 
corresponding to the hyperplane
$$H_x = \{ f \in L \mid f(x) = 0 \} \subset L.$$
The assumption that $\lb$ is very ample implies that the Kodaira map $\Phi_L$ is an embedding.

Alternatively we can describe $\Phi_L$ more concretely as follows. Let
$x \in X$. Let $h \in L$ be a section of $\lb$ with $h(x) \neq 0$;
such an $h$ exists since $\lb$ is very ample (hence $L$ has
no base locus). Let $\ell_{x} \in L^*$ be 
defined by: 
$$\ell_{x} (f) = f(x)/h(x), \quad \forall f \in L.$$
Then $\Phi_L(x)$ is the point in $\p(L^*)$ represented by $\ell_x$.
It is straightforward to check that $\Phi_L(x)$ is independent 
of the choice of the section $h$.

Now let $L^k$ denote the image of the $k$-fold product $L \otimes
\cdots \otimes L$ in $H^0(X, \lb^{\otimes k})$ under the natural map
given by taking the 
product of sections. (In general this map may not be surjective.) The
homogeneous coordinate ring of $X$ with respect to the 
embedding $\Phi_L: X \hookrightarrow \p(L^*)$ can be identified with the graded algebra 
$$R = R(L) = \bigoplus_{k \geq 0} R_k,$$ where $R_k := L^k$. This is a subalgebra of the {\it ring of sections} 
$$R(\lb) = \bigoplus_{k \geq 0} H^0(X, \lb^{\otimes k}).$$

Recall that the \emph{Hilbert function} of the graded algebra $R$ is
defined as $H_R(k) := \dim_{\c}(R_k) = \dim_{\c}(L^k)$. The
celebrated theorem of Hilbert on the degree of a
projective variety states the following: 
\begin{enumerate} 
\item For sufficiently large values of $k$, the function $H_R(k)$
coincides with a polynomial of degree $n = \dim_{\c}(X)$. 
\item Let $a_n := \lim_{k
  \to \infty} H_R(k)/k^n$ be the leading coefficient of this
polynomial.  Then $n! a_n$ is equal to the degree of 
the projective embedding of $X$ in 
$\p(L^*)$ (in other words, the self-intersection number
of the divisor class of the line bundle $\lb$).
\end{enumerate}

\subsection{Valuations and Newton-Okounkov bodies} \label{subsec-valuation}

Newton-Okounkov bodies give information about asymptotic behavior of Hilbert functions of graded algebras. In this 
section we review some background material about Newton-Okounkov bodies following \cite{KKh1, KKh2}. 
In the present paper we will only be concerned with 
graded algebras which are homogeneous coordinate rings of projective
varieties, as introduced in the previous section, so in the discussion below we restrict to 
this case. 

Following the notation of the previous section, 
we now wish to associate a convex body
$\Delta(R) \subset \r^n$ to $R$. In the case when $X$ is a projective
toric variety, there is a well-known such convex body - namely, the
Newton polytope of $X$, {that is, the convex hull of the weights of $T$ acting on $L$}. 
However, for a general projective variety, we
do not have such toric methods at our disposal. The tool we use for
this purpose - as initially proposed by Okounkov in \cite{Okounkov1,
  Okounkov2} and further developed in \cite{KKh1, KKh2, LM} - is a \textbf{valuation} on the
field $\c(X)$ of rational functions on the variety $X$. We now recall
the definition. 

We fix once and for all a total order $<$ on the lattice $\z^n$ which
respects addition.
Indeed, for additional concreteness, we always consider the standard
lexicographic order. 

\begin{Def} \label{def-valuation}
A {\it valuation} on the field $\c(X)$ is a function 
$v: \c(X) \setminus\{0\} \to \z^n$ satisfying the following: for any $f, g \in \c(X) \setminus \{0\}$,
\begin{itemize}
\item[(a)] $v(fg) = v(f) + v(g)$,
\item[(b)] $v(f+g) \geq \min(v(f), v(g))$, and 
\item[(c)] $v(\lambda f) = v(f)$, for all $0 \neq \lambda \in \c$.
\end{itemize}
Moreover, we say the valuation $v$ has {\it one-dimensional leaves}
if it additionally satisfies the following: 
\begin{equation}\label{one-dimensional}
\begin{minipage}{0.7\linewidth}
if $v(f) = v(g)$, then
there exists a non-zero constant $\lambda \neq 0 \in \c$ such that 
$v(g - \lambda f) > v(g)$ or $g-\lambda f = 0$. 
\end{minipage}
\end{equation}

\end{Def}

If $v$ is a valuation with one-dimensional leaves, then the image of $v$ is a sublattice of $\z^n$ of full rank. 
Hence, by replacing $\z^n$ with this sublattice if necessary, we will always
assume without loss of generality that $v$ is surjective.

Given a variety $X$, there exist many possible valuations with
one-dimensional leaves on its field of rational functions 
$\c(X)$. Below, we list several examples which arise quite naturally
in geometric contexts. 

\begin{Ex} \label{ex-Grobner-val}
Let $X$ be a curve and let $p$ be a smooth point on $X$.  
Then \[
v(f) := \textup{ order of zero or pole of $f$ at $p$} 
\]
defines a valuation. 
More generally, if $X$ is an $n$-dimensional variety, a 
choice of a coordinate system at a smooth point $p$ on $X$ 
gives a valuation on $\c(X)$ with one-dimensional leaves.
\end{Ex}

\begin{Ex} \label{ex-Parshin-val} Generalizing
  Example~\ref{ex-Grobner-val} further, we can construct a valuation
  out of a flag of subvarieties in $X$. More specifically, let
$$\{p\} = X_n \subset \cdots \subset X_0 = X $$ be a sequence of closed irreducible subvarieties of $X$
such that $\dim_\c(X_k) = n-k$ for $0 \leq k \leq n$, and assume $X_{k}$ is
non-singular along $X_{k+1}$ for $0 \leq k < n$ (i.e. the local ring $\mathcal{O}_{X_{k}, X_{k+1}}$ is regular). 
Such a sequence of subvarieties is
sometimes called a \emph{Parshin point} on the variety $X$.  A collection $u_1, \ldots, u_n$
of rational functions on $X$ is a {\it system of parameters} about
such a sequence if for each $k$, ${u_k}_{|X_k}$ is a rational function
on $X_k$ which is not identically zero and which has a zero of first
order on the hypersurface $X_{k+1}$ (in other words the image of $u_k$ is a generator of the maximal ideal of $\mathcal{O}_{X_k, X_{k+1}}$). By the non-singularity assumption
above, such a system of parameters always exists.  Given a sequence of
normal subvarieties and a system of parameters $u_1, \ldots, u_n$, we
can define a valuation $v$ on $\c(X)$ with one-dimensional leaves and
values in $\z^n$ as follows. Let $f \in
\c(X)$ with $f \neq 0$. Then $v(f) = (k_1, \ldots, k_n)$ where the
$k_i$ are defined inductively: the first coordinate $k_1$ is the order of vanishing of $f$ on $X_1$. Then $f_1
= (u_1^{-k_1}f)_{|X_1}$ is a well-defined rational function on $X_1$
which is not identically zero, and $k_2$ is the order of vanishing of
$f_1$ on $X_2$. Continuing in this manner defines all the $k_i$. (In fact, 
the normality assumption on the $X_k$ is not crucial; it can be
avoided by passing
to the normalizations.)
\end{Ex}

The following proposition is simple but fundamental; it states that the image of
the valuation on a finite-dimensional subspace $E$ of $\c(X)$ is in one-to-one correspondence with a basis of $E$, in
much the same way that the integral points in the Newton polytope
of a projective toric variety $X$ correspond to a basis of
$H^0(X,\lb)$. 
The proof is straightforward from the 
defining properties of a valuation with one-dimensional leaves. 

\begin{Prop} \label{prop-val-dim}
Let $E \subset \c(X)$ be a finite-dimensional subspace of $\c(X)$.
Then $\dim_\c(E) = \#v(E \setminus \{0\})$. In other words, the
dimension of $E$ is equal to the number of values which the valuation
$v$ attains on $E \setminus \{0\}$. 
\end{Prop}

Fix a choice of a valuation $v$ with one-dimensional leaves on $\c(X)$. 
Using the valuation $v$ we now associate a semigroup $S(R) \subset \n \times \z^n$
to the homogeneous coordinate ring $R$ of $X$. 
First we identify $L = H^0(X, \lb)$ with a (finite-dimensional)
subspace of $\c(X)$ by choosing a non-zero element $h \in L$ and
mapping $f \in L$ to the rational function $f/h \in \c(X)$. Similarly,
we can associate the rational function $f/h^k$ to an element $f \in R_k := L^k \subseteq H^0(X,
\lb^{\otimes k})$. Using these
identifications, we define 
\begin{equation}\label{eq:definition S}
S = S(R) = S(R,v,h) = \bigcup_{k > 0} \{ (k, v(f / h^k)) \mid f \in
L^k \setminus \{0\}\}.
\end{equation}
From the property (a) in Definition~\ref{def-valuation} it follows
that $S(R)$ is an additive semigroup. Moreover, 
from the property~\eqref{one-dimensional} in Definition \ref{def-valuation} it is
straightforward to show the following. 

\begin{Prop}
The group generated by the semigroup $S = S(R)$, considered as a
semigroup in $\z \times \z^n \cong \z^{n+1}$, is (all of) $\z^{n+1}$.
\end{Prop}

\begin{Rem} \label{rem-depend-valuation}
The semigroup $S = S(R)$ depends on the 
choice of valuation $v$ on $\c(X)$ and the section $h$. The dependence
on $h$ is minor; a different choice of $h'$ would lead to a semigroup
which is shifted by the vector $k v(h/h')$ at the level $\{k\} \times \z^n$. 
However, the dependence 
on the valuation $v$ is much more subtle. 
\end{Rem}

In order to keep track of the natural $\n$-grading on the ring $R =
\bigoplus_{k \geq 0} R_k$, 
it is convenient to extend the valuation $v$ to a valuation $\tilde{v}: R \setminus \{0\} \to \n \times \z^n$ as follows.
We define an ordering on $\n \times \z^n$ by $(m, u) \leq (m', u')$ if
and only if 
\begin{equation}\label{eq:ordering on N times Z^n} 
\textup{ either ($m>m'$) \quad or \quad ($m=m'$ and $u \leq u'$) (note the switch!) }
\end{equation}
For $f \in R$, we now define
\begin{equation}\label{eq:def tilde v} 
\tilde{v}(f) := (m, v(f_m/h^m))
\end{equation}
where $f_m$ is the highest-degree homogeneous component of $f$. The map $\tilde{v}$ is a valuation on $R$
with the above ordering on $\n \times \z^n$. By construction, the
image of the valuation $\tilde{v}$ is exactly the semigroup $S = S(R)$.

Given $S \subset \n \times \z^n$ an arbitrary additive semigroup, we
can associate to it a convex body as follows. Let $C(S)$ denote  
the closure of the convex hull of $S \cup \{0\}$, considered as a
subset of $\r_{\geq 0} \times \r^n \subset \r^{n+1}$. It is a closed convex cone with apex at the
origin. If $S$ is a finitely generated semigroup then 
$C(S)$ is a {rational polyhedral} cone. Consider the intersection of the cone $C(S)$ with the plane $\{1\} \times \r^n$ and 
let $\Delta(S)$ denote the projection of this intersection to $\r^n$, via the projection on the second factor $\r \times \r^n \to \r^n$. 
Equivalently, $\Delta(S)$ can be described as 
\begin{equation}\label{eq:def Delta S}
\Delta(S) = \overline{\conv(\bigcup_{k>0} \{x/k \mid (k, x) \in S
  \})}.
\end{equation}
If the cone $C(S)$ intersects the plane $\{0\} \times \r^n$ only at the origin, then the convex set $\Delta(S)$
is bounded and hence is a convex body. This will be the case for all
the semigroups we consider in this
paper. We refer to $\Delta(S)$ as the \textbf{convex body associated to the 
semigroup $S$.}

\begin{Def} \label{def-Ok-body}
Let $S = S(R)$ be the semigroup associated to $(R, v, h)$ as above.
We denote the cone $C(S)$ and the convex body $\Delta(S)$ by 
$C(R)$ and $\Delta(R)$ respectively. In this case, the cone $C(R)$ intersects $\{0\} \times \r^n$ only at the origin and 
hence $\Delta(R) = \Delta(R,v,h)$ is a convex body, the \textbf{Okounkov body (also called Newton-Okounkov body) of
  $(R,v,h)$}. It is also sometimes denoted as $\Delta(X,v)$ or simply
$\Delta(X)$ when we wish to emphasize the underlying projective variety $X$. 
\end{Def}

\begin{Rem} 
From Remark~\ref{rem-depend-valuation} it follows that a different choice $h' \in L$ 
would yield a Newton-Okounkov body $\Delta(R,v,h')$ which is shifted by the fixed vector $v(h/h')$. 
Thus, similar to the case of $S(R)$, the dependence of the Newton-Okounkov body on the choice of section $h$ is minor.
However, as in Remark~\ref{rem-depend-valuation}, the dependence of $\Delta(R,v)$ 
on the valuation $v$ is more subtle. 
\end{Rem}

The Newton-Okounkov body $\Delta(R)$ encodes information about the 
asymptotic behavior of the Hilbert function of $R$ (see \cite{KKh1, KKh2}, \cite{LM} and 
\cite{Okounkov1, Okounkov2}).
Let $H_R(k) := \dim_\c(R_k)$ be the Hilbert function of the graded algebra $R$.

\begin{Th} \label{th-asymp-Hilbert-Okounkov-body}
The Newton-Okounkov body $\Delta(R)$ has real dimension $n$, and 
the leading coefficient $$a_n = \lim_{k \to \infty} \frac{H_R(k)}{k^n},$$ 
of the Hilbert function of $R$ is equal to $\Vol_n(\Delta(R))$, the
Euclidean volume of $\Delta(R)$ in $\r^n$. In particular, the degree
of the projective embedding of $X$ 
in $\p(L^*)$ is equal to $n! \Vol_n(\Delta(R))$. 
\end{Th}

\begin{Rem} \label{rem-Okounkov-body} In \cite{KKh2}, the asymptotic
  behavior of Hilbert functions of a much more general class of graded
  algebras is addressed, namely, graded algebras of the form $A = \bigoplus_{k \geq 0} A_k$ such that 
\begin{itemize} 
\item $A_k \subset \c(X)$ for all $k$, and
\item $A$ is
  contained in a finitely generated algebra of the same form.
\end{itemize}
Note that such an algebra $A$ is not necessarily finitely generated.
This class of graded algebras includes arbitrary \emph{graded linear
series} on a variety, and in particular, the ring of sections of arbitrary
line bundles (not just ample ones).
\end{Rem}

\section{A toric degeneration associated to a valuation} \label{sec-toric-degen}

As in Section \ref{sec-prelim-part2}, let $X$ be an $n$-dimensional
projective variety with a very ample line bundle $\lb$, the
corresponding homogeneous coordinate ring $R$, a valuation $v$ with
one-dimensional leaves and the value semigroup $S = S(R, v, h)$. 
Moreover, from now on, we place
the additional assumption that: 
\begin{center}  
\emph{$S$ is finitely generated.} 
\end{center} 
This assumption implies that the Newton-Okounkov body $\Delta(R)$ is a
rational polytope. The semigroup algebra $\c[S]$ of a semigroup $S
\subseteq \n \times \z^n$ is defined to be the subalgebra of $\c[t,
x_1, \ldots, x_n]$ spanned by the monomials $t^k x_1^{a_1} \cdots
x_n^{a_n}$ for all $(k, a_1,\ldots, a_n) \in S$. Recall that $S$ is
the image of the extended valuation $\tilde{v}: R \setminus \{0\} \to
\n \times \z^n$ of~\eqref{eq:def
  tilde v}. 
In \cite{KKh1}, the authors mentioned briefly 
that the homogeneous coordinate ring $R$ of $X$ can be
degenerated in a flat family to the semigroup algebra $\c[S]$
associated to the value semigroup $S = S(R)$. In fact, Teissier proved already in 1999 that  
there always exists a toric degeneration from an algebra to the
semigroup algebra associated to a valuation \cite{Te99}. (More
generally, the existence of a degeneration of an algebra to its
associated graded with respect to a filtration is 
well-known and classical in commutative algebra \cite[Chapter 6]{Eisenbud}.) 
In the context of Newton-Okounkov bodies, this idea was 
developed further, with suggestions on applications to other areas such as Schubert calculus, 
in \cite{Anderson10}. Geometrically, the existence of the algebraic degeneration means that 
the variety $X$ can be degenerated, in a flat family $\X$, to the (not necessarily normal)
projective toric variety $X_0$ associated to the semigroup $S$. The
normalization of $X_0$ is then the (normal) toric variety
corresponding to the polytope $\Delta(R)$.

Our construction of an integrable system uses Anderson's toric
degeneration as a key ingredient, so in this section we both review and
slightly expand on his exposition in \cite{Anderson10}. 
First we recall how the semigroup algebra $\c[S]$ of $R$ is related to
the data of the extended valuation $\tilde{v}$. The link comes from a
filtration of $R$ defined using $\tilde{v}$ as follows. Let 
\[
R_{\geq (m,u)} := \left\{f \in R \, \lvert \, \tilde{v}(f) \geq (m,u)
  \textup{ or } f=0\right\}
\]
and similarly for $R_{>(m,u)}$. By the definition of the valuation $v$
and the definition of the order~\eqref{eq:ordering on N times Z^n} on
$\n \times \z^n$, each $R_{\geq (m,u)}$ (and $R_{>(m,u)}$) is a
(finite-dimensional) vector subspace of $R$. Moreover, again by
properties of valuations, for any $(m,u)$ and $(m', u')$, we have 
\[
R_{\geq (m,u)} \cdot R_{\geq (m', u')} \subseteq R_{\geq (m+m',
  u+u')}.
\]
We can now define the associated graded ring $\gr R$ with respect to
this filtration as 
\[
\gr R := \bigoplus_{(m,u)} R_{\geq (m,u)} / R_{> (m,u)}.
\]
This is naturally an $S$-graded ring. In fact, since $\tilde{v}$ has
one-dimensional leaves, the spaces $R_{\geq (m,u)}/R_{>(m,u)}$ have
dimension $1$ precisely when $(m,u) \in S$, and $0$ otherwise.
Since the homogeneous elements of $\gr R$ are not zero divisors, 
one shows that $\gr R$ is isomorphic to the semigroup algebra $\c[S]$ (see \cite[Remark 4.13]{BG}).

As mentioned above, there is a well-known method, namely the Rees
algebra construction, of degenerating an algebra $R$ to its
associated graded $\gr R$ corresponding to a filtration of $R$ in a
flat family. This is explained in more detail in \cite{Anderson10}.
Alternatively, one can describe the same degeneration using a choice
of a collection of sections 
$\{f_{ij}\} \in R$ as we now explain. This choice gives rise to an
embedding of the degenerating family in a weighted projective space,
which we use in our geometric constructions. 

\begin{Def}\label{def-Khovanskii}
Let $R = \bigoplus_i R_i$ and $\tilde{v}$ as above. 
We say that a set $\BB \subset R$ is a {\bf Khovanskii basis} 
\footnote{The name ``Khovanskii basis'' was suggested by B. Sturmfels in honor of A. G. Khovanskii's influential 
contributions to combinatorial algebraic geometry.} 
{\bf for $(R, \tilde{v})$} if 
the set of images $\tilde{v}(\BB) = \{\tilde{v}(f) \mid f \in \BB\}$ generate the semigroup
$S = S(R) = S(R,v, h)$. 
\end{Def}
This definition generalizes the notion of a SAGBI (Subalgebra Analogue of
Gr\"obner Basis for Ideals) basis for a subalgebra of a polynomial ring, and
was introduced in \cite{Kaveh-string} and \cite{Manon}. In general a Khovanskii basis may or may not be finite.

Since we assumed that $S$ is finitely generated $R$ has a finite Khovanskii basis $\BB$.
Without loss of generality and by adding more elements to $R$ if necessary, we can assume that $\BB$ consists of 
homogeneous elements $f_{ij}$, $1 \leq i \leq r$, $1 \leq j \leq n_i$, 
with the following additional properties: 
\begin{equation*}
\begin{minipage}{0.9\linewidth}
(a) $f_{ij} \in R_i$ for all $1 \leq i \leq r, 1 \leq j \leq n_i$, and 
\end{minipage}
\end{equation*}
\begin{equation*}
\begin{minipage}{0.9\linewidth}
(b) for each $i$, the collection $\{f_{i1}, f_{i2}, \ldots, f_{i
  n_i}\}$ is a vector space (i.e. additive) basis for $R_i$. 
\end{minipage}
\end{equation*}

For the remainder of this discussion, we fix a finite Khovanskii basis 
$\{f_{ij}\}_{1 \leq i \leq r, 1 \leq j \leq n_i}$ of $R$ satisfying
the conditions (a) and (b) above. 
Since the $\tilde{v}(f_{ij})$ generate the semigroup $S(R)$, by 
a standard argument using the classical \emph{subduction algorithm}, 
we can show that every element of $R$ can be represented as a
polynomial in the $f_{ij}$. For the sake of
completeness we include a proof. 

\begin{Prop} \label{prop-subduction}
Any $f \in R$ can be written as a polynomial in the $f_{ij}$. In other words, 
a finite Khovanskii basis $\{f_{ij}\}$ of $R$ generates $R$ as an algebra. 
\end{Prop}

\begin{proof}
  First we observe that in the value semigroup $S = S(R, \tilde{v})$
  every increasing sequence of elements eventually stabilizes. Let $(m_1, u_1) \leq
  (m_2, u_2) \leq \cdots$ be an increasing sequence in $S$. Then from
  the definition of $\leq$ on $\n \times \z^n$ (see \eqref{eq:ordering
    on N times Z^n}) we have $m_1 \geq m_2 \geq \cdots$. Since $\n$ is
  well-ordered and each $S_i$ is finite we see that for $i$
  sufficiently large $(m_i, u_i) = (m_{i+1}, u_{i+1}) = \cdots$ as
  claimed. Now take $0 \neq h \in R$. Write $\tilde{v}(h) =
  \sum_{i=1}^t k_i \tilde{v}(f_i)$. Since $v$ (and hence $\tilde{v}$)
  have one-dimensional
 leaves, there exists a nonzero constant $\lambda \in \c$ such that $g_1 := h - c f_1^{k_1} \cdots f_t^{k_t}$
has the property that either $g_1 = 0$ or $\tilde{v}(g_1) >
\tilde{v}(h)$. If $g_1=0$ we are done. Otherwise, 
repeating the same argument for $g_1$ in place of $h$, we obtain a sequence 
$\tilde{v}(h) < \tilde{v}(g_1) < \tilde{v}(g_2) < \cdots$. Since an
increasing sequence in $S$ eventually stabilizes, we see that at some
point we must have
$g_i = 0$. This finishes the proof. 

\end{proof}

We can now state the main result of this section \cite[Proposition
5.1]{Anderson10}.

\begin{Th} \label{th-toric-degen}
Let $R$ be as above and assume $S = S(R)$ is finitely generated. Then 
there is a finitely generated, $\n$-graded, flat $\c[t]$-subalgebra $\RR \subset R[t]$, such that:
\begin{itemize}
\item[(a)] $\RR[t^{-1}] \cong R[t, t^{-1}]$ as $\c[t, t^{-1}]$-algebras, and
\item[(b)] $\RR / t \RR \cong \gr R$.
\end{itemize}
\end{Th}

The geometric interpretation \cite[Corollary 5.3]{Anderson10} of
Theorem~\ref{th-toric-degen} is obtained by taking $\Proj$ with
respect to the $\n$-grading on $\RR$. 

\begin{Cor} \label{cor-toric-degen}
There is a projective flat family $\pi: \X := \Proj~ \RR \to \c$ such that:
\begin{itemize}
\item[(a)] For any $z \neq 0$ the fiber $X_z = \pi^{-1}(z)$ is isomorphic to $X = \Proj~ R$. 
More precisely, the family over $\c \setminus \{0\}$, i.e. 
$\pi^{-1}(\c \setminus \{0\})$ is isomorphic to $X \times (\c \setminus \{0\})$. 
\item[(b)] The special fiber $X_0 = \pi^{-1}(0)$ is isomorphic to
  $\Proj(\gr R) \cong \Proj \c[S]$ and is equipped with an action of
  $(\c^*)^n$, where $n=\dim_\c X$. 
The normalization of the variety 
$\Proj(\gr R)$ is the toric variety associated to the rational polytope $\Delta(R)$.
\end{itemize}
\end{Cor}

\begin{proof}[Sketch of proof of Theorem \ref{th-toric-degen}]

Let $\tilde{v}(f_{ij}) = (i, u_{ij})$.  Let $\overline{f}_{ij}$
  denote the image of $f_{ij}$ in the associated graded $\gr R$. Let $A =
  \c[x_{ij}]$ denote the polynomial algebra in the indeterminates
  $x_{ij}$, $1 \leq i \leq r, 1 \leq j \leq n_i$. Define an
  $(\n \times \z^n)$-grading on $A$ by $\deg(x_{ij}) := (i,u_{ij})$; thus the
  surjective map $A \to \gr R$ defined by $x_{ij} \mapsto
  \overline{f}_{ij}$ is a map of graded rings. The kernel of this map
  is a homogeneous ideal $I_0$. Let $\overline{g}_1, \ldots,
  \overline{g}_q$ be homogeneous generators for the kernel and let 
  $\deg(\overline{g}_k) = (n_k, v_k)$. It follows
  that $\overline{g}_k(f_{11}, \ldots , f_{rn_r})$ lies in
  $R_{>(n_k,v_k)}$ for each k.  By the proof of Proposition
  \ref{prop-subduction} one can find elements $g_k \in \overline{g}_k
  + A_{<(n_k, v_k)}$ such that $g_k(f_{11},\ldots ,f_{rn_r}) = 0$. The
  $g_k$ will not be homogeneous for the full $\n \times \z^n$ grading
  of $A$, but since the $f_{ij}$ are homogeneous for the first $\n$
  factor, the $g_k$ can be chosen to respect the $\n$-grading as
  well. 

The induced map $A/\langle g_1, \ldots, g_q\rangle \to R$ is
  an isomorphism. In fact, if $I$ denotes the kernel of the map $A \to
  R$, then $g_k \in I$ by construction, and the initial terms $g_k$
  generate the initial ideal $I_0$ of $I$ (which is the kernel of $A
  \to \gr R$). It follows that the $g_k$'s form a Gr\"{o}bner basis for
  $I$, with respect to the term order determined by the order on $\z
  \times \z^n$, cf. \cite[Exercise 15.14(a)]{Eisenbud}. (Specifically, the term order is $\prod_{ij}
  x_{ij}^{a_{ij}} \leq \prod_{ij}x_{ij}^{b_{ij}}$ if and only if
  $\sum_{ij} a_{ij}(i, u_{ij}) \leq \sum_{ij} b_{ij}(i, u_{ij})$. Note
  that we allow ties between monomials in this notion of term order.)

The remainder of the construction and proof will only be sketched, see
\cite{Anderson10} for details. Here we record only the points directly
relevant to our construction in
Section~\ref{sec-int-system-from-val}. 
Let $\mathcal{S} \subset \z \times \z^n$ be a finite subset consisting of all the degrees $(s, v)$ for all the 
monomials appearing in any of the $g_k$. Let $p: \z \times \z^n \to \z$ be a linear map which preserves the 
ordering on $\mathcal{S}$ (cf. \cite[Lemma 5.2, Proof of Prop. 5.1]{Anderson10}).
Let $w_{ij} = p(i, u_{ij})$ be the degree of $x_{ij}$ under the weighting induced by $p$. From the choice of $p$ it follows that 
for any $k$, the initial term of $g_k$ with respect to the weighting defined by $p$ is exactly $\overline{g}_k$. Let $\ell_k = p(n_k ,v_k)$, and set 
$$\tilde{g}_k = \tau^{\ell_k} g_k(\tau^{-w_{11}}x_{11} ,\ldots,\tau^{-w_{rn_r}}x_{rn_r}) \in A[\tau],$$
where $\tau$ is an indeterminate. 
Consider the homomorphism $A[\tau] \to R[t]$ given by
\begin{equation}\label{eq:homomorphism Atau to Rt}
x_{ij} \mapsto t^{w_{ij}}f_{ij} \quad \textup{ and } \tau \mapsto t.
\end{equation}
Let $\RR$ be the image of this homomorphism; this ring satisfies all the
conditions of the theorem.  In addition, the ring $\RR$ can be
presented explicitly as $\RR =
A[\tau]/(\tilde{g}_1,\ldots,\tilde{g}_q)$. (This is a standard way of
producing a Rees ring from a Gr\"{o}bner basis; see, e.g.,
\cite[Theorem 15.17]{Eisenbud}.)  
\end{proof} 

The results just recounted yield a projective flat family $\X$
associated to $R$ and $S(R,v,h)$. In fact, we will see below, in
Section~\ref{sec-kahler}, that the choice of a Khovanskii basis
$\{f_{ij}\}$ above naturally gives rise to an explicit embedding of the family
into $W\p \times \c$, where $W\p$ denotes a weighted projective space,
and the restriction of the standard projection $W\p \times \c \to \c$
to the family $\X$ is exactly the projection $\pi: \X \to \c$ of
Corollary~\ref{cor-toric-degen}. By embedding the weighted projective
space $W\p$ into a ``large'' projective space $\p({\bf V}_d^*)$ we
then obtain an embedding of the family $\X$ into $\p({\bf
  V}_d^*) \times \c$. Our construction of an integrable system on the original
variety $X$ will depend on this embedding in the sense that we use a
K\"ahler structure on $\X$ which is the restriction of a K\"ahler structure on $\p({\bf
  V}_d^*) \times \c$, judiciously chosen to be compatible with the
original K\"ahler structure on $X$ as well as the
(compact) torus-invariant K\"ahler structure on the toric variety
$X_0$. The details are the subject of the next section. 

As a final remark, we note that if $X$ is assumed to be smooth, then
by Corollary \ref{cor-toric-degen}(a) we know that the family $\X$ is
nonsingular away from the special fibre $X_0$. In fact, by Proposition
\ref{prop-sing-family} we know that more is true. We have the
following.

\begin{Cor} \label{cor-sing-family} 
Let $X, R, \tilde{v}, S$ and $\X$ be as above. Suppose that
$X$ is smooth. Then the family $\X$ is nonsingular away from the singular locus of the special fiber $X_0$.
\end{Cor}

\section{A K\"ahler structure on the family $\X$} \label{sec-kahler} 
As usual let $X$ be an $n$-dimensional projective variety. From now on in addition we assume that
\begin{center}
\emph{$X$ is smooth.}
\end{center}
Let $\lb$ be a very ample line bundle on $X$ with the corresponding embedding $X \hookrightarrow \p(L^*)$ where 
$L = H^0(X, \lb)$. We fix a K\"ahler structure $\omega$ on $X$ by restrcting a Fubini-Study K\"ahler structure 
$\Omega_H$ on the projective space $\p(L^*)$ with respect to a Hermitian product $H$ on $L^*$
(see e.g. \cite{GH}).

We follow the notation and assumptions of Section \ref{sec-toric-degen}, in particular that the value semigroup $S$ is 
finitely generated. As mentioned at the end of Section~\ref{sec-toric-degen}, an explicit
choice of a finite Khovanskii basis $\{f_{ij}\}$ naturally gives rise to an
embedding of the toric degeneration $\X$ constructed in
Corollary~\ref{cor-toric-degen} into $\p({\bf V}_d^*) \times \c$ for
an appropriate projective space $\p({\bf V}_d^*)$ (to be defined
concretely below). This embedding allows us to build an extension of
the K\"ahler structure on the original variety $X$ (the generic fiber
of $\X$) to all of $\X$, in such a way that the extension restricts to
a torus-invariant K\"ahler structure on the special fiber $X_0$. We
now explain the details. 

\subsection{Weighted projective spaces}\label{subsec:wps}
Let $R$ and $S(R,v,h)$ be as in Section~\ref{sec-toric-degen}. In
general, a Khovanskii basis $\{f_{ij}\}$ has elements 
of $R$ of different homogeneous degrees. Thus the choice of
$\{f_{ij}\}$ naturally leads to an embedding of the family $\X$ into a
\emph{weighted} projective space $W\p$ (we recall the definition below). In the simple case when the Khovanskii
basis consists only of elements of degree $1$, the weighted projective
space is an ordinary projective space. In addition, in the next
section we embed the weighted projective space into a large (ordinary)
projective space $\p({\bf V}_d^*)$. Thus it is not strictly logically necessary to
discuss the weighted projective spaces (since we may simply embed $X$
directly into $\p({\bf V}_d^*)$) but we find it more natural to first
discuss the map $X \to W\p$ which takes account of the degrees of the
$f_{ij}$. 

We first recall the definition and some basic facts about weighted
projective spaces. 
Let $V_1, \ldots, V_r$ be finite dimensional vector spaces over $\c$ and $m_1, \ldots, m_r$ positive integers. Consider the 
action of $\c^*$ on $V_1 \times \cdots \times V_r$ by:
$$\lambda \cdot (v_1, \ldots, v_r) = (\lambda^{m_1} v_1, \ldots, \lambda^{m_r} v_r).$$ 
We define the {\bf weighted projective space} $W\p = W\p(V_1, \ldots, V_r; m_1, \ldots, m_r)$ to be the 
quotient of $(V_1 \times \cdots \times V_r) \setminus \{(0, \ldots,
0)\}$ by this $\c^*$-action.
The point in $W\p$ 
represented by $(v_1, \ldots, v_r) \in V_1 \times \cdots \times V_r$
is denoted 
$(v_1 : \cdots : v_r)$. There is a natural rational map $W\p \ratmap \p(V_1) \times \cdots \times \p(V_r)$ given by:
$$(v_1: \cdots : v_r) \mapsto ([v_1], \ldots, [v_r]).$$ This is defined whenever all the $v_i$ are nonzero.
The following is well-known.

\begin{Th} \label{th-basic-prop-WPS}
The weighted projective space $W\p$ is a normal projective
variety. Moreover, $W\p$ is smooth  if and only if $m_1 = \cdots =
m_r$.
\end{Th}

\subsection{Embedding a projective variety into a weighted projective space} \label{subsec-embed-in-WP}
Returning now to the general setup of Section~\ref{sec-prelim-part2}, 
let $X$ be a projective variety equipped with a very ample line bundle $\lb$ and let $L = H^0(X, \lb)$. 
Let $r > 0$ be a positive integer. For any $k$ with $1 \leq k \leq r$,
define $V_k := ({L^k})^*$ and $m_k := k$. 
Consider the 
associated weighted projective space
$$W\p = W\p(L^*, (L^2)^*, \ldots, (L^r)^*; 1, 2, \ldots, r).$$ 

Analogous to the construction of the Kodaira map $\Phi_L$, we can also define a morphism ${\Psi}: X \to W\p$ as follows. 
Let $x \in X$. Let $h \in L$ be a section with $h(x) \neq 0$; as
before, since
$L$ has no base locus, such an $h$ always exists. 
For each $k$ with $1 \leq k \leq r$, let $\ell_{x, k}$ be the element in $(L^k)^*$ defined by
$$\ell_{x, k} (f) := f(x)/h^k(x), \quad \forall f \in L^k.$$
Define ${\Psi}(x)$ to be the point $(\ell_{x, 1} : \cdots : \ell_{x,
  r}) \in W\p$. It can be checked that ${\Psi}(x)$ is independent 
of the choice of the section $h$ and is a morphism from $X$ to $W\p$.
Note that each $\Phi_{L^k}$ is an embedding, so $\Psi$ is also 
an embedding. 

The morphism $\Psi$ can also be described explicitly in terms of
coordinates (at least on an open dense subset) as follows. 
Fix bases $\{f_{i1}, \ldots, f_{in_i}\}$ for the $L^i$. With respect
to the corresponding 
dual bases for the $({L^i})^*$), the map $\Psi$ can be given in
coordinates by
\begin{equation} \label{equ-embed-in-WP}
{\Psi}: x \mapsto \left( \left( \frac{f_{11}(x)}{h(x)}, \ldots , \frac{f_{1n_1}(x)}{h(x)}  \right)
: \cdots : \left( \frac{f_{r1}(x)}{h^r(x)} , \ldots , \frac{f_{rn_r}(x)}{h^r(x)} \right) \right)
\end{equation}
at the points $x \in X$ where $h(x) \neq 0$.

With this notation in place, we take a moment to concretely and geometrically interpret the family
$\X$ constructed in Section \ref{sec-toric-degen} as a subvariety of $W\p \times \c$; this embedding is 
determined by the map~\eqref{eq:homomorphism Atau to
  Rt}. Specifically, consider the map
$$\rho: X \times \c^* \to \X \hookrightarrow W\p \times \c = W\p(L^*, (L^2)^*,
\ldots, (L^r)^*; 1, 2, \ldots, r) \times \c$$ given by
\begin{equation} \label{equ-family}
(x, t) \mapsto \left( \left( (t^{w_{11}} \frac{f_{11}(x)}{h(x)}, \ldots, t^{w_{1n_1}}\frac{f_{1n_1}(x)}{h(x)}) : \cdots :
    (t^{w_{r1}} \frac{f_{1r}(x)}{h^r(x)}, \ldots, t^{w_{rn_r}} \frac{f_{rn_r}(x)}{h^r(x)}) \right), t \right)
\end{equation}
where $h$ is a section with $h(x) \neq 0$ (see \eqref{equ-embed-in-WP}).
The image of the family $\X$ sitting inside the weighted projective
space $W\p \times \c$ is then the closure of the
image of this map. This image is cut out by the equations
$\tilde{g}_k(x_{ij}, \tau) = 0$ for $k=1, \ldots, q$. The projection 
$\pi: \X \to \c$ of Corollary~\ref{cor-toric-degen} is then the
projection $\X \subseteq
W\p \times \c \to \c$ to the second factor.

\begin{Rem} \label{rem-zero-fiber} 
By construction, the image of $X_0 = \pi^{-1}(0)$
in the weighted projective space is cut out by the equations
$\overline{g}_k(x_{ij}) = 0$ for $1 \leq k \leq q$. The map
$\c[x_{ij}] \to \c[x_{ij}] / I_0 \cong \gr R$ is a homomorphism of
$(\n \times \z^n)$-graded algebras and hence the embedding $X_0 \hookrightarrow
W\p \times \{0\}$ is $\T=(\c^*)^n$ equivariant. 
\end{Rem}

\subsection{Embedding a weighted projective space in a large projective space} \label{subsec-embed-WP-in-big-proj}
In order to construct a K\"ahler structure on the family $\X$, it will be useful to embed the
(usually singular) weighted projective space $W\p$ into a ``large'' classical
(and hence smooth) projective space. In this section we describe
a particular such embedding. 

Let $V_1, \ldots, V_r$ be finite-dimensional vector spaces over $\c$
and $m_1, \ldots, m_r$ positive integers
as in Section~\ref{subsec:wps}. 
Let $\Sym((V_1 \times \cdots \times V_r)^*)$ denote the algebra of
polynomials on $V_1 \times \cdots \times V_r$. 
We equip $\Sym((V_1 \times \cdots \times V_r)^*) \cong \Sym(V_1^*
\times V_2^* \times \cdots \times V_r^*)$ with a grading by assigning
degree $m_i$ to any $f_i^* \in V_i^*$. 
Let $d>0$ be a positive integer. With respect to the above
grading, let ${\bf V}_d$ denote the subspace of  
$\Sym((V_1 \times \cdots \times V_r)^*)$ consisting of the elements of
total degree $d$.
From the definition of the grading it then follows that 
\begin{equation} \label{equ-Vd}
{\bf V}_d \cong \bigoplus_{\stackrel{\beta=(\beta_1, \ldots, \beta_r)
    \in \z^r_{\geq 0}}{\sum_i m_i \beta_i = d}} \bigotimes_{i=1}^r\Sym_{\beta_i}(V_i^*)
\end{equation}
where $\Sym_{\beta_i}(V_i^*)$ denotes the space of homogeneous
polynomials of degree $\beta_i$ on $V_i$. 

Given $v = (v_1, \ldots, v_r) \in V_1 \times \cdots \times V_r$, the evaluation map
$$\ell_v: p \mapsto p(v)$$ defines a linear function $\ell_v$ on $\Sym((V_1 \times \cdots \times V_r)^*)$.
It follows that the map $v \mapsto {\ell_v}_{|{\bf V}_d}$ induces a rational map
$$\Theta: W\p(V_1, \ldots, V_r) \ratmap \p({\bf V}_d^*).$$

Moreover, we have the following \cite{Dolgachev}.

\begin{Prop} \label{prop-Dolgachev}
If $d = \prod_{i=1}^r m_i$ then $\Theta$ is a morphism and an embedding.
\end{Prop}

The map $\Theta$ can also be described concretely in
coordinates. Suppose $n_i = \dim_\c(V_i)$ for each $i$ and 
fix a basis for each $V_i$; let 
$(x_{i1}, \ldots, x_{in_i})$ denote the corresponding coordinates for
$V_i$ (i.e. dual vectors in $V_i^*$). 
Then, under the identification $\Sym((V_1 \times \cdots \times V_r)^*)
\cong \Sym(V_1^* \times \cdots \times V_r^*) \cong 
\c[x_{ij}; 1 \leq i \leq r, 1 \leq j \leq n_i]$, the set of monomials
\begin{equation}\label{eq:set of monomials}
\left\{x_{11}^{\alpha_{11}} \cdots x_{rn_r}^{\alpha_{rn_r}} \bigg\lvert
\sum_{i=1}^r m_i \sum_{j=1}^{n_i} 
\alpha_{ij} = d,~ \alpha_{ij} \in \z_{\geq 0} \right\}
\end{equation}
form a basis for ${\bf V}_d$. With respect to these coordinates, 
the morphism $\Theta: W\p \to \p({\bf V}_d^*)$ can be described
explicitly as 
\begin{equation} \label{equ-WP-embed}
\Theta: ((x_{11}, \ldots, x_{1n_1}): \cdots :(x_{r1}, \ldots,
x_{rn_r})) \mapsto \left(x_{11}^{\alpha_{11}} \cdots
  x_{rn_r}^{\alpha_{rn_r}} \bigg\lvert  \sum_{i=1}^r \sum_{j=1}^{n_i} m_i
  \alpha_{ij} = d,~\alpha_{ij} \in \z_{\geq 0} \right).
\end{equation}

The following will be used in Section~\ref{sec-GIT-part2}. 

\begin{Rem} \label{rem-T-equiv}
Let $\HH \cong (\c^*)^m$ be an algebraic torus. 
Suppose $\HH$ acts linearly on each $V_i$. The diagonal $\HH$-action on 
$V_1 \times \cdots \times V_r$ induces an action of $\HH$ on
$W\p(V_1, \ldots, V_r; m_1, \ldots, m_r)$ and on $\p({\bf V}_d^*)$. It
is not hard to see that $\Theta$ is $\HH$-equivariant with respect to these actions.
\end{Rem}

Consider now the special case
in Section~\ref{subsec-embed-in-WP},
so $V_k = (L_k)^*$ and $m_k = k$.  In this setting, ${\bf V}_d$ is the subspace of 
$\Sym(L \times \cdots \times L^r) \cong \Sym((L^* \times \cdots \times (L^r)^*)^*)$ consisting of homogeneous elements of 
degree $d$. 
Let $ d= r!$. Proposition \ref{prop-Dolgachev} then implies
that the composition $\Theta \circ \Psi: X \to \p({\bf V}_d^*)$ is an
embedding. Again we take a moment to describe $\Theta \circ \Psi$
explicitly in terms of coordinates. As before,
let $x \in X$ and let $h \in L$ such that $h(x) \neq 0$. Using the
same bases $\{f_{i1}, \ldots, f_{i n_i}\}$ for $L^i$ as above, 
note that any $p \in {\bf V}_d$ can be written as 
\[
p = \sum_{\stackrel{\alpha=(\alpha_{11}, \ldots, \alpha_{1n_1}, \ldots,
  \alpha_{r1}, \ldots, \alpha_{rn_r}), \alpha_{ij} \in \z_{\geq
    0}}{\sum_{i=1}^r i \left( \sum_{j=1}^{n_i} \alpha_{ij} \right) = d
  = r!}} c_\alpha f_{11}^{\alpha_{11}}
\cdots f_{1n_1}^{\alpha_{1n_1}} \cdots f_{r1}^{\alpha_{r1}} \cdots
f_{rn_r}^{\alpha_{rn_r}}.
\]
From the previous coordinate descriptions of $\Theta$ and $\Psi$ it is
now straightforward to see that $(\Theta \circ \Psi)(x) \in \p({\bf
  V}_d^*)$ can be described as the point 
\begin{equation}\label{equ-Theta-Psi}
\left( \frac{f_{11}(x)^{\alpha_{11}} \cdots
    f_{rn_r}(x)^{\alpha_{rn_r}}}{h(x)^d} \; \bigg\lvert \; \sum_{i=1}^r i
  \left( \sum_{j=1}^{n_i} \alpha_{ij} \right) = d = r! \; , \; \alpha_{ij}
  \in \z_{\geq 0} \right)
\end{equation}
on the locus where $h(x) \neq 0$. 

For each $i > 0$, the sum and product of sections give a natural linear map from $L^{\otimes i}$ onto $L^i$. Similarly
there is a natural linear map from ${\bf V}_d$ onto $L^d$. These induce embeddings $(L^{i})^* \hookrightarrow (L^{\otimes i})^*$ and 
$(L^d)^* \hookrightarrow {\bf V}_d^*$, which in turn give embeddings
of the corresponding projective spaces 
$\p((L^i)^*) \hookrightarrow \p((L^{\otimes i})^*)$ and
$j: \p((L^d)^*) \hookrightarrow \p({\bf V}_d^*)$. From
\eqref{equ-Theta-Psi} and the definition of the Kodaira map $\Phi_{L^d}$, 
it follows that the diagram
\begin{equation} \label{equ-Theta-commutative}
\xymatrix{
X \ar[r]^{\Psi} \ar[rd]_{\Phi_{L^d}} & W\p \ar[r]^{\Theta} & \p({\bf V}_d^*) \\
& \p((L^d)^*) \ar[ru]_{j} & \\
}
\end{equation}
commutes.

The following will be used in Section~\ref{sec-GIT-part2}.

\begin{Rem} \label{rem-T-line-bundle}
Suppose $\HH$ is an algebraic torus. Suppose $\HH$ acts on $X$ and that
the very ample line bundle $\lb$ is 
$\HH$-linearized. Then, for any $k>0$, the spaces
$(L^k)^*$ are $\HH$-modules and hence $\HH$ also acts on $W\p$ and
$\p({\bf V}_d^*)$ as in Remark \ref{rem-T-equiv}.
It can be checked that 
all the maps in the diagram \eqref{equ-Theta-commutative} are $\HH$-equivariant for the corresponding $\HH$-actions.
\end{Rem}

\subsection{A K\"ahler structure on $\p({\bf V}_d^*)$} \label{subsec-Kaehler}
In order to define a K\"ahler structure $\tilde{\omega}$ on the family
$\X$, we first choose an appropriate
K\"ahler structure on $\p({\bf V}_d^*)$. 
Recall that the K\"ahler structure $\omega$ on $X$ is defined as the pull-back of a
Fubini-Study form $\Omega_H$ on $\p(L^*)$ with respect to a fixed Hermitian product $H$ on $L^*$. 
In this section we wish to construct a K\"ahler structure on the 
projective space $\p({\bf V}_d^*)$ such that its pull-back under
the embedding $\Theta \circ \Psi$ coincides with the form $\omega$ on $X$.

For this purpose, it is convenient to extend the diagram
\eqref{equ-Theta-commutative}. 
We maintain the notation of Sections~\ref{subsec-embed-WP-in-big-proj} and \ref{subsec-embed-in-WP}.
Suppose $\beta = (\beta_1, \ldots, \beta_r) \in \z^r_{\geq 0}$ where
$\sum_i i \cdot \beta_i = d$, as in
Section~\ref{subsec-embed-WP-in-big-proj}. Identifying $L^{\otimes d}
\cong \bigotimes_{i=1}^r L^{\otimes i \cdot \beta_i}$ in the natural way,
consider the map 
\begin{equation}\label{eq:L tensor d to Sym beta_i}
  L^{\otimes d} \to \bigotimes_{i=1}^r \Sym_{\beta_i}(L^i)
  \end{equation}
induced by taking the natural map $L^{\otimes i \cdot \beta_i} \to
\Sym_{\beta_i}(L^i)$, induced by the product of sections, on each
factor. Taking the direct sum of~\eqref{eq:L tensor d to Sym beta_i}
over all such $\beta$ we obtain a natural map
\begin{equation}
  \label{eq:sum beta L tensor d to Vd}
  \bigoplus_{\stackrel{\beta = (\beta_1, \ldots, \beta_r) \in
      \z^r_{\geq 0}}{\sum_{i=1}^r i \cdot \beta_i = d}} L^{\otimes d}
  \to {\bf V}_d
\end{equation}
using the identification~\eqref{equ-Vd}. Since~\eqref{eq:sum beta L
  tensor d to Vd} is induced from taking the product of sections, it
is straightforward to see that~\eqref{eq:sum beta L tensor d to Vd}
fits into a commutative diagram 
$$
\xymatrix{
\bigoplus_{\stackrel{\beta = (\beta_1, \ldots, \beta_r) \in
      \z^r_{\geq 0}}{\sum_{i=1}^r i \cdot \beta_i = d}} L^{\otimes d} \ar[r] \ar[d] & {\bf V}_d \ar[d]\\
L^{\otimes d} \ar[r] & L^d.\\
}
$$
where the left vertical map takes the sum of the components in the
direct sum, the top horizontal arrow is~\eqref{eq:sum beta L tensor d
  to Vd}, and the other two arrows are the natural ones induced from
the sum and product of sections. It is clear that all four maps are
surjective. 

Taking duals, we get the diagram of inclusions
\begin{equation} \label{equ-comm-diag-duals}
\xymatrix{
{\bf V}_d^* \ar@{^{(}->}[r] & \bigoplus_{\stackrel{\beta = (\beta_1, \ldots, \beta_r) \in
      \z^r_{\geq 0}}{\sum_{i=1}^r i \cdot \beta_i = d}} (L^{\otimes d})^*\\
(L^d)^* \ar@{^{(}->}[r] \ar@{^{(}->}[u] & (L^{\otimes d})^* \ar@{^{(}->}[u]\\
}
\end{equation}
where the right vertical map is given by the diagonal inclusion $v
\mapsto (v,v, \ldots, v)$. 
Considering the induced morphisms between the corresponding projective spaces and putting this together with 
\eqref{equ-Theta-commutative}, we obtain the commutative diagram
\begin{equation} \label{equ-big-comm-diag}
\xymatrix{
X \ar[d]^{\Phi_L} 
\ar[r]^{\Psi} \ar[rrd]_{\Phi_{L^d}} & W\p \ar[r]^{\Theta} & \p({\bf V}_d^*) \ar[r] & \p(\bigoplus_\beta (L^{\otimes d})^*)\\
\p(L^*) & & \p((L^d)^*) \ar[r] \ar[u] & \p((L^{\otimes d})^*) \ar[u]\\
}
\end{equation}
where all the maps are embeddings. 

We will use the Hermitian product $H$ on $L^*$ to define a Hermitian
product ${\bf H}$ on the vector space $\bigoplus_\beta (L^{\otimes
  d})^*$.  First we recall the following. Let $H_1$, $H_2$ be
Hermitian products on vector spaces $V_1$, $V_2$ respectively. Then
$H_1 \oplus H_2$ and $H_1 \otimes H_2$, defined respectively by the
formulas 
\begin{eqnarray} \label{equ-sum-prod-Hermitian}
(H_1 \oplus H_2)((v_1, v_2), (w_1, w_2)) &=& H_1(v_1, w_1) + H_2(v_2,
w_2) \cr 
(H_1 \otimes H_2)(v_1 \otimes v_2, w_1 \otimes w_2) &=& H_1(v_1, w_1)H_2(v_2, w_2),
\end{eqnarray}
for all $v_i, w_i \in V_i$, yield well-defined Hermitian products on
$V_1 \oplus V_2$ and $V_1 \otimes V_2$ respectively.
Recalling also that $(L^{\otimes d})^* \cong (L^*)^{\otimes d}$, we
now define 
\begin{equation}
  \label{eq:def big H}
  {\bf H} := \bigoplus_{\stackrel{\beta = (\beta_1, \ldots, \beta_r) \in
      \z^r_{\geq 0}}{\sum_{i=1}^r i \cdot \beta_i = d}} H^{\otimes d} 
\end{equation}
on $\bigoplus_\beta L^{\otimes d}$.  By restriction ${\bf H}$ induces
Hermitian products on the other three spaces in the diagram
\eqref{equ-comm-diag-duals}.  The 
Hermitian product on $(L^{\otimes d})^* \cong (L^*)^{\otimes d}$
induced from the diagonal inclusion $v \mapsto (v,v,\ldots, v)$ 
coincides with the Hermitian product $N \cdot H^{\otimes d}$ where
here $N$ denotes the number of components in the direct sum
$\bigoplus_\beta L^{\otimes d}$.  We now define 
\begin{equation}
\label{eq:def big Omega}
{\bf \Omega} := \frac{1}{N \cdot d} \Omega_{\bf H}
\end{equation} where $\Omega_{\bf H}$ is the K\"ahler form on
$\p(\bigoplus_\beta (L^{\otimes d})^*)$ corresponding to the 
Hermitian product ${\bf H}$. 

The normalization factor in~\eqref{eq:def big Omega} is chosen
precisely so that the following holds. 

\begin{Th} \label{th-Kaehler-form}
The pull-back of the form ${\bf \Omega}$ on $\p(\bigoplus_\beta
(L^{\otimes d})^*)$ to $X$, via the maps in~\eqref{equ-big-comm-diag}, 
coincides with the original K\"ahler form $\omega$ with respect to the
Hermitian product $H$ on $L^*$. 
\end{Th}

\begin{proof}
By what was said above, the pull-back of ${\bf \Omega}$ to $\p((L^d)^*)$ is equal to 
$(1/N \cdot d) \Omega_{H^{\otimes d}}$, where $\Omega_{H^{\otimes d}}$
denotes the K\"ahler form on $\p((L^d)^*)$ induced from
$H^{\otimes d}$. 
{The theorem now follows from the fact that 
the pull-back of $\Omega_{H^{\otimes d}}$ to $X$ under the Kodaira map $\Phi_{L^d}$
is equal to $d \cdot \omega$ (see \cite[Lemma 7.9]{KKh3}).} Recall $\omega$ is the pullback of
$\Omega_H$ to $X$ under $\Phi_L$. Since $\Omega_H$ pulls back to
$\p((L^d)^*)$ to equal $N \cdot \Omega_{H^{\otimes d}}$, the claim
follows from~\eqref{eq:def big Omega}. 
\end{proof}

Finally, we address the invariance of the K\"ahler form ${\bf \Omega}$ under a torus action.
Let $\T \cong (\c^*)^n$ and let $T = (S^1)^n$ denote the usual compact
torus in $\T$.
Suppose $\T$ acts linearly on each space 
$(L^i)^*$ such that the Hermitian product
$H^{\otimes i}$ on $(L^i)^*$ is $T$-invariant.
A natural geometric situation when this arises is 
when $\T$ acts on $X$ and $\lb$ is a $\T$-linearized line bundle
equipped with a $T$-invariant Hermitian form.  
Note that the $\T$-actions on the $(L^i)^*$ induce a $\T$-action on
${\bf V}_d^*$.

\begin{Prop} \label{prop-inv-Hermitian-V_d}
{Suppose $T$ acts linearly on the $L^i$ and suppose for each $i$ the $T$-action on $L^i$ preserves
$H^{\otimes i}$. Then the Hermitian product ${\bf H}$ on ${\bf V}_d^*$ is
$T$-invariant. In particular, 
the corresponding K\"ahler form ${\bf \Omega}$ on $\p({\bf V}_d^*)$ is $T$-invariant.} 
\end{Prop}

\begin{proof}
Let $T$ act on a vector space $V$ equipped with a Hermitian product $H$. Then 
the $T$-action preserves $H$ if and only if there exists an
orthonormal basis of $V$ consisting of 
$T$-weight vectors.
For each $i$, let $\{ f_{i1}^*, \ldots, f_{in_i}^*\}$ be a basis
consisting of $T$-weight vectors. 
By assumption, the $T$-action preserves $H^{\otimes i}$, so we may
assume without loss of generality that the basis is orthonormal.
Under the identification $\Sym((L^1)^*\otimes (L^2)^* \otimes \cdots
\otimes (L^r)^*) \cong \c[f_{ij}; 1 \leq i \leq r, 1 \leq j \leq n_i]$
(compare~\eqref{eq:set of monomials} and~\eqref{equ-WP-embed}), the
set of vectors 
\begin{equation} \label{equ-basis-V_d}
{\bf B} = \left\{ \prod_{i, j} (f_{ij}^*)^{\alpha_{ij}} \, \bigg\lvert  \, \alpha=(\alpha_{ij}); 
\sum_{i=1}^r i \cdot \sum_{j=1}^{n_i} \alpha_{ij} = d; \alpha_{ij} \in
\z_{\geq 0}) \right\}
\end{equation}
is a basis of $T$-weight vectors for ${\bf V}_d^*$. 
Moreover, from the definition of ${\bf H}$, 
it follows that ${\bf B}$ is an orthonormal basis, as desired. 
\end{proof}

\subsection{K\"ahler structure on the family $\X$} \label{subsec-Kaehler-family}
We are finally prepared to construct the K\"ahler structure on the
toric degeneration $\X$ constructed in Section~\ref{sec-toric-degen}. 
Recall that $X$ is a smooth projective variety equipped with 
$\lb$ a very ample line bundle and we set $L = H^0(X,\lb)$ and let $R$
denote its
homogeneous coordinate ring. Let $v$ be a valuation with
one-dimensional leaves on $\c(X)$ and assume that the associated
semigroup $S = S(R)$ is finitely
generated. 
Also let $\Delta = \Delta(R)$ denote the associated Newton-Okounkov body. 
Since $S$ is finitely generated, $\Delta$ is a rational polytope.

In Section \ref{sec-toric-degen} we described a flat family $\X$ such that general fibers $X_z$, $z \neq 0$, 
are isomorphic to $X$ and the central fiber $X_0$ is the (possibly non-normal) toric variety associated to the semigroup $S$. 
Recall from Section~\ref{subsec-embed-in-WP} (in particular the map
$\rho$ described in~\eqref{equ-family}) that the family $\X$ lies in $W\p \times \c$ where the weighted projective space $W\p$ is
$W\p(L^*, \ldots, (L^r)^*; 1, \ldots, r)$ for some $r>0$. By embedding
$W\p$ in a usual projective space $\p({\bf V}_d^*)$ as in Section
\ref{subsec-embed-WP-in-big-proj}, we may therefore think of the family
$\X$ as a subvariety in 
the smooth variety $\p({\bf V}_d^*) \times \c$. Let $\Omega$ denote
the K\"ahler structure on $\p({\bf V}_d^*)$ constructed in
Section~\ref{subsec-Kaehler}. As in Part \ref{part-1}
we equip $\p({\bf V}_d^*) \times \c$ with the product symplectic structure
$\Omega \times \left( \frac{i}{2} dz \wedge d\bar{z} \right)$ where
$z$ is the complex parameter on the $\c$ factor and so $\frac{i}{2} dz
\wedge d\bar{z}$ is the standard K\"ahler form on $\c$. We denote
by $\tilde{\omega}$ (respectively $\omega_z$) the restriction of this product structure to the
smooth locus of the family $\X$ (respectively the fiber $X_z$). 
For $z \neq 0$, $(X_z, \omega_z)$ is a
smooth K\"ahler manifold. When we write $\omega_0$, we mean the
K\"ahler form on the smooth locus of $X_0$. 

Recall that the construction of the family $\X$ involved a choice of 
elements $\{f_{ij}\}$ in $R$, satisfying properties
(a)-(c) in Definition \ref{def-Khovanskii}. 
The following lemma proves that we may also assume, without loss of
generality, that in addition to the properties (a)-(c), the collection
$\{f_{ij}\}$ is compatible with the choice of a Hermitian metric. 

\begin{Lem} \label{lem-valuation-orthonomal}
Let $v: \c(X) \setminus \{0\} \to \z^n$ be a valuation with one-dimensional leaves.
Let $V \subset \c(X)$ be a finite-dimensional vector subspace 
and let $H$ be a Hermitian product on the dual space $V^*$. Then there exists a basis $\{b_1, \ldots, b_k\}$ 
of $V$ such that 
\begin{enumerate} 
\item the corresponding dual basis $\{b_1^*, \ldots, b_k^*\}$ is orthonormal with respect to $H$ and
\item $v(b_i) \neq v(b_j)$ for all $i \neq j$, i.e., 
$v(V \setminus \{0\}) = \{v(b_1), \ldots, v(b_k)\}$. 
\end{enumerate}
\end{Lem}

\begin{proof}
From Proposition~\ref{prop-val-dim} we know that $\# v(V \setminus
\{0\}) = \dim_\c V$. Let $k=\dim_\c V$ and let 
$\{a_1 > \cdots > a_k\} = v(V \setminus \{0\})$, listed in decreasing
order. 
For $a \in \z^n$ let $V_{ \geq a}  := \{ x \mid v(x) \geq a\} \cup \{0\}$. This is a subspace of $V$.  We have a flag of subspaces
$$\{0\} \subsetneqq V_{\geq a_1} \subsetneqq \cdots \subsetneqq V_{\geq a_k} =
V.$$ 
Using the Hermitian product on $V^*$, we define an identification
$\phi: V^* \to V$ which thus also equips $V$ with a Hermitian product
by pullback via $\phi$. In such a situation an orthonormal basis of
$V$ is identified via $\phi$ with its own dual basis in $V^*$, which
is also orthonormal. Thus, the desired basis can be constructed by
finding an orthonormal basis of $V_{\geq a_1}$, then inductively an
orthonormal basis for $V_{\geq a_s}$ extending the given one on
$V_{\geq a_{s-1}}$. 
\end{proof}

From this lemma it follows that we may assume that our Khovanskii basis 
$\{f_{ij}\}$ has the property that, for each fixed $i, 1 \leq i \leq
r$, the collection $\{f_{ij}\}_{1 \leq j \leq n_i}$ is a basis for
$L^i$ such that its dual basis $\{f^*_{i1}, \ldots, f^*_{i n_i}\}$ is
orthonormal in $(L^i)^*$, with respect to the Hermitian product on
$(L^i)^*$ induced from $H^{\otimes i}$. 

Now let the torus $\T$ act linearly on each $(L^i)^*$ such that every
$f_{ij}^*$ is a $T$-weight vector with weight $u_{ij} =
v(f_{ij})$. Note this corresponds to the action on $W\p(L^*,
(L^2)^*, \ldots, (L^r)^*; 1,2, \ldots, r)$ (see Remark \ref{rem-T-equiv}). 
From the orthogonality of the basis $\{f_{i1}^*, \ldots,
f_{in_i}^*\}$ it follows that the induced action of $T$ on
$(L^i)^*$ preserves the Hermitian product. 

Recall that for $0 \neq z \in \c$, $\rho_z$ denotes the 
restriction of the embedding $\rho: X \times \c^* \to \X \subset W\p
\times \c$ given in
\eqref{equ-family} to $X \times \{z\}$. 

\begin{Prop} \label{prop-omega-t} We have:
\begin{itemize}
\item[(a)] Under the isomorphism $\rho_1: X \to X_1 \subset \X$, the
  pullback $\rho_1^* \omega_1$ of the K\"ahler form $\omega_1$ on 
$X_1$ is equal to the original K\"ahler form $\omega$ on $X$.
\item[(b)] The K\"ahler form $\omega_0$ on the (smooth locus of) the $\T$-toric variety $X_0 \subset \p({\bf V}_d^*) \times \{0\} \cong \p({\bf V}_d^*)$ 
is invariant under the $T$-action.
\end{itemize}
\end{Prop}
\begin{proof}
The embedding $$\rho_1: X  \hookrightarrow W\p \times \{1\}$$ coincides with the embedding 
$\Psi: X \to W\p$ in diagram \eqref{equ-Theta-commutative}. This proves (a).
Part (b) follows from Proposition \ref{prop-inv-Hermitian-V_d} where $\T$ acts on each space $L^i$ in such a way that the $f_{ij}$
are weight vectors with weights $u_{ij}$.
\end{proof}

\section{Integrable systems constructed from a valuation} \label{sec-int-system-from-val}
We may now state the main theorem of Part~\ref{part-2} of
this manuscript. The results from the previous sections make the
proof quite straightforward. 
We use the setting of Section~\ref{subsec-Kaehler}, namely that $X$ is an $n$-dimensional complex projective variety
equipped with $\lb$ a very ample Hermitian line bundle on $X$, with corresponding K\"ahler form $\omega$ obtained via the Kodaira 
embedding. We let $R$ denote the 
ring of sections of the line bundle $\lb$. We take a valuation 
$v: \c(X) \setminus \{0\} \to \z^n$ with one-dimensional leaves, where $\z^n$ is equipped with a total order. Then 
$S=S(R)$ and $\Delta = \Delta(R)$ denote the value semigroup and the Newton-Okounkov body of $(R, v)$ respectively. In this setting
we have the following. 
\begin{Th} \label{th-B}
With notation as above, suppose that $(X, \omega)$ is smooth and the semigroup $S = S(R)$ is finitely generated. 
\begin{itemize}
\item[(1)] There exists a completely integrable system $\mu: (F_1, \ldots, F_n) \to \r^n$ on $X$ in the sense of Definition \ref{def-int-system}, i.e.: 
\begin{itemize}
\item[(a)] The functions $F_1, \ldots, F_n$ are continuous on all of $X$.
\item[(b)] There exists an open dense subset $U$ 
such that $F_1, \ldots, F_n$ are differentiable on $U$, 
and the differentials $dF_1, \ldots, dF_n$ are linearly
independent on $U$. %removed $U$ contained in the smooth locus of $X$.
\item[(c)] $F_1, \ldots, F_n$ pairwise Poisson-commute on $U \subset X$. 
\end{itemize}
Moreover, the image of $\mu$ coincides with the Newton-Okounkov body $\Delta = \Delta(R)$.
\item[(2)] The integrable system in (1) generates a torus 
action on the dense open subset $U$ of $X$. Moreover, the inverse image of the interior of 
$\Delta$ lies in the open subset $U$.
\end{itemize}
\end{Th}
\begin{proof}
Corollary \ref{cor-toric-degen}, Corollary \ref{cor-sing-family} and Proposition \ref{prop-omega-t} show that the conditions 
(a)-(d) in Theorem \ref{th-A} are satisfied for the toric degeneration
constructed from the valuation $v$,
equipped with the K\"ahler structure $\tilde{\omega}$ built in Section~\ref{subsec-Kaehler-family}. Thus Theorem \ref{th-B}
follows from Theorem \ref{th-A}.
\end{proof}

\section{GIT and symplectic quotients}  \label{sec-GIT-part2}
In this section we discuss the compatibility of the flat family
constructed in Section \ref{sec-toric-degen} (out of the data of a
valuation with a finitely generated value semigroup) with a torus
action. We will show that the family can be constructed so that the
compatibility conditions in Section \ref{sec-GIT-part1} are satisfied
and hence the integrable system from Section
\ref{sec-int-system-from-val} descends to the quotient by the
torus. The material in this section is straightforward, but we spell
out the details for completeness. 

\subsection{Quotients of toric varieties} \label{subsec-GIT-toric}
$\T = (\c^*)^n$ denotes, as before, the $n$-dimensional algebraic
torus with lattice of characters $\z^n$.  {Consider a finitely generated semigroup $S \subset \n \times \z^n$ of integral points.
We will assume that $S$ generates the whole lattice $\z^{n+1} = \z \times \z^n$ and also $S_1 = S \cap (\{1\} \times \z^n) \neq \emptyset$. 
The semigroup $S$ gives rise to the semigroup algebra $\c[S]$. The algebra $\c[S]$ is the subagebra of the algebra of polynomials
$\c[t, x_1, \ldots, x_n]$ consisting of polynomials whose monomials belong to the semigroup $S$. That is, 
$$\c[S] = \{ f \in \c[t, x_1, \ldots, x_n] \mid  f = \sum_{(k, \alpha) \in S} c_{(k, \alpha)} t^k x^\alpha\}.$$
Here as usual $x^\alpha$ is shorthand for $x_1^{a_1} \cdots x_n^{a_n}$. 
The projection of $S$ on the first factor gives a grading on the algebra $\c[S]$ (by non-negative integers) and the 
corresponding variety $\Proj(\c[S])$ is a projective $\T$-toric variety. 
}

Now suppose $\HH \subset \T$
is an algebraic subtorus of dimension $m$ with lattice of characters
$M$. The inclusion $\HH \subset \T$ gives a surjective homomorphism
$\lambda: \z^n \to M$, which in turn extends to a linear map
$\lambda_\r: \r^n \to M_\r = M \otimes \r$.  Let $\tilde{\lambda}:
\z^{n+1} \to M$ be a homomorphism such that the restriction of 
$\tilde{\lambda}$ to $\{0\} \times
\z^n$ coincides with $\lambda$. Fix $(1, a_0) \in S$ and define 
$\lambda_0 := \tilde{\lambda}(1, a_0)$. Then for any $(k,
a) \in \z^{n+1}$, the fact that $\tilde{\lambda}$ is a homomorphism
implies 
$$\tilde{\lambda}(k, a) = \lambda(a - ka_0) + k\lambda_0$$ so the lift
$\tilde{\lambda}$ is determined by $\lambda$ and by its value at
$(1,a_0)$. We will return to this point later. 
The lifted homomorphism $\tilde{\lambda}$ gives an $M$-grading on the semigroup algebra $\c[S]$ where for $(k, a) \in S$, the monomial
$t^kx^a$ has degree $\tilde{\lambda}(k, a)$. We can uniquely extend $\tilde{\lambda}$ to a linear function
$\tilde{\lambda}_\r: \r^{n+1} \to M_\r$.
Let $\Lambda'_\r \subset \r^{n+1}$  denote the kernel of $\tilde{\lambda}_\r$ and let $S' = S \cap \Lambda'_\r$ be 
the subsemigroup of $S$ obtained by intersecting with the subspace $\Lambda'_\r$. The kernel of the linear map $\lambda_\r$ 
has dimension $n-m$ and hence the subspace $\Lambda'_\r$ has dimension $n-m+1$.
\begin{Lem} \label{prop-S'}
Let $S$ and $S'$ be as above, and let $\Delta = \Delta(S)$, $\Delta' = \Delta(S')$ be the convex bodies associated to 
the semigroups $S$ and $S'$ respectively (see \eqref{eq:def Delta S}). Then: 
\begin{itemize}
\item[(a)] $S'$ is a finitely generated semigroup. 
\item[(b)] $(\{1\} \times \Delta') = (\{1\} \times \Delta) \cap
  \Lambda'_\r$. 
\end{itemize}
\end{Lem}

\begin{proof}
Part (b) of the Lemma follows immediately from the definitions. Part
(a) is a consequence of Gordan's Lemma (see e.g. \cite[Proposition
2.17]{CoxLittleSchenck}). 
\end{proof}

The $M$-grading on $\c[S]$ gives rise to an $\HH$-action
on the $\T$-toric variety $X_S := \Proj~ \c[S]$. By construction, 
the $\HH$-invariant subalgebra $\c[S]^\HH$, which is the homogeneous-degree-$0$ part 
of $\c[S]$ with respect to the $M$-grading, is the semigroup algebra $\c[S']$.  
Let $X_S' := \Proj~ \c[S']$. The algebra $\c[S']$ is $\n \times \z^n$ graded and hence the variety $X_S'$ has a $\T$-action. By definition the variety $X'_S$ 
is the GIT quotient of the variety $X_S$ with respect to the $\HH$-action. 
In fact, it is a toric variety for the action of a torus $\T'$ which is
a quotient of $\T$ (by a subgroup containing, but possibly larger than, $\HH$). The complex dimension of the variety $X_S'$ (and  the torus $\T'$) is 
equal to the real dimension of the polytope $\Delta'$. 

{Fix a set of generators $(i, u_{ij})$ for $S$. Embed $X_S$ in a weighted projective space $W\p$. Let 
$W\p \hookrightarrow \p(V_d^*)$ be embedding of this weighted projective space in a large projective space.  
We equip $X_S \hookrightarrow \p(V_d^*)$ with a symplectic structure as in Section~\ref{sec-prelim-part1}.} 
Let $\mu_H: X_S \to \Lie(H)^* = M_\r$, $\mu_T: X_S \to \Lie(T)^* \cong \r^n$ denote the moment maps
for the actions of the compact tori $H \subset \HH$ and $T \subset
\T$, respectively, 
on $X_S$. 
Recall that by the Kempf-Ness theorem the GIT quotient $X'_S$ can also
be realized as the 
symplectic quotient $\mu_H^{-1}(0) / H$ of $X_S$ at $0$, {provided that
$0$ is a regular value of $\mu_H$ and $H$ acts freely on $\mu_H^{-1}(0)$}. The GIT
quotient $X'_S$ inherits a symplectic (in fact K\"ahler) 
structure coinciding with the quotient 
symplectic structure on $X'_S$ \cite[Section 8.3]{MKF}. 
Let $\mu_{T'}: X'_S \to \Lie(T')^*$ denote the moment map of $X'_S$ regarded as a Hamiltonian $T'$-space. Then the following diagram
\begin{equation} \label{equ-GIT-toric-moment}
\xymatrix{
& \Delta_H \ar@{}[r]|{\subset} & \Lie(H)^*\\
X_S \ar[r]^{\mu_T} \ar[ru]^{\mu_H} & \Delta \ar@{}[r]|{\subset} \ar[u]_{\lambda_\r} & \Lie(T)^* \\
\mu_H^{-1}(0) \ar@{^{(}->}[u]^i \ar[d]^p & \\
X'_S \ar[r]^{\mu_{T'}} & \Delta' \ar@{}[r]|{\subset} \ar@{^{(}->}[uu] & \Lie(T')^*\\
}
\end{equation}
commutes, 
where $i: \mu_H^{-1}(0) \hookrightarrow X_S$ is the inclusion map, and $p: \mu_H^{-1}(0) \to \mu_H^{-1}(0)/H = X'_S$ is 
the symplectic quotient map. Moreover, the image of $\Lie(T')^*$ in $\Lie(T)^*$ lies in the kernel of the linear map $\lambda_\r$.

\subsection{Invariant valuations}\label{subsec:invariant val}

We now consider a more general situation. As usual let $\HH
\cong (\c^*)^m$ denote an $m$-dimensional algebraic torus with
character group $M \cong \z^m$. 
Let $X$ be a projective $\HH$-variety of dimension $n$. Note in
particular that we do not assume $X$ is an $\HH$-toric variety, and it
may be that $m$ is strictly less than $n$. The $\HH$-action on $X$ induces an $\HH$-action on the field of rational functions 
$\c(X)$.

Suppose $\lb$ is an $\HH$-linearized
very ample line bundle on $X$. Then $L = H^0(X, \lb)$ is a finite-dimensional $\HH$-module and the homogeneous coordinate
ring $R = R(L)$ is a graded $\HH$-algebra.
We define the \textbf{weight semigroup} of $R$ to be 
$$S_\HH(R) := \{(k, \lambda) \, \mid \, \textup{ there exists $f \in
  R_k$ with $t \cdot f = t^\lambda f$ for all $t \in \HH$}  \} \subset \n
  \times M.$$ 
If $R$ is finitely generated as an algebra, then $S_\HH(R)$ is a
finitely generated semigroup. 
Hence the convex body $\Delta_\HH := \Delta(S_\HH(R))$ associated
via~\eqref{eq:def Delta S} to the semigroup $S_\HH(R)$ is a rational
polytope. Following \cite{Brion} and \cite{KKh4}, we call $\Delta_\HH$ 
the {\it moment polytope} of the $\HH$-algebra $R$.
The terminology is motivated from the well-known fact that the polytope
$\Delta_\HH$ coincides with the (closure of the) image of the moment map of 
$X_{smooth}$ regarded as a Hamiltonian $H$-space with respect to an
$H$-invariant K\"ahler structure on $\p(L^*)$ (here $H \subset \HH$
denotes the maximal compact torus in $\HH$).

We now fix an $\HH$-invariant valuation $v: \c(X) \setminus \{0\} \to
\z^n$ with one-dimensional leaves. 
Such a valuation always exists (see \cite{Okounkov1, KKh4}).
Let $h \neq 0$ in $L$ be a $\T$-weight vector of weight $\lambda_0$. 
Following the method in Section~\ref{subsec-valuation}, we now define
$\tilde{v}: R \setminus \{0\} \to \n \times \z^n$ by 
\begin{equation}\label{eq:def tilde v Tinvariant}
\tilde{v}(f) = (k, v(f/h^k))
\end{equation}
for $f \in R_k \setminus \{0\}$. Also as before, $\tilde{v}$ has one-dimensional
leaves. The following is straightforward.

\begin{Lem} \label{prop-tilde-v-invariant}
The valuation $\tilde{v}$ in~\eqref{eq:def tilde v Tinvariant} is $\HH$-invariant. 
\end{Lem}

Let $S = S(R) \subset \n \times \z^n$ be the semigroup associated to the algebra $R$ and the valuation $\tilde{v}$. 
As before, we assume that $S(R)$ is 
finitely generated and that, as a group, $S$ generates all of
$\z^{n+1} \cong \z \times \z^n$. 
We record the following \cite{Okounkov1, KKh4}. 

\begin{Lem} \label{prop-tilde-lambda}
Let $(k, a) \in S$. 
\begin{itemize}
\item[(a)] There exists an $\HH$-weight vector $f \in R_k$ with $\tilde{v}(f) = (k, a)$.
\item[(b)] Let $\lambda$ be the $\HH$-weight of the vector $f$ found in
  part (a). The association $(k, a) \mapsto \lambda$ gives a
  well-defined function which is a semigroup homomorphism 
$S \to M$.  
\end{itemize}
\end{Lem}

Since $S$ generates the group $\z^{n+1}$, the semigroup homomorphism
constructed in Lemma~\ref{prop-tilde-lambda}(b) extends uniquely to a group homomorphism
$$\tilde{\lambda}: \z^{n+1} \to M.$$
Let $\lambda = \tilde{\lambda}_{|\{0\} \times \z^n}$ denote the restriction
of $\tilde{\lambda}$ to $\{0\} \times \z^n \cong \z^n$. 
As in Section \ref{subsec-GIT-toric} we can recover $\tilde{\lambda}$ 
from its restriction $\lambda$ as
follows. Notice that since $\tilde{v}(h) = (1, 0)$, we have
$\tilde{\lambda}(1,0) = \lambda_0$. Then for any $(k, a) \in S$ we have 
$\tilde{\lambda}(k, a) = \lambda(a) + k\lambda_0$.
Also as in Section~\ref{subsec-GIT-toric}, we extend $\tilde{\lambda}$ to a linear map $\tilde{\lambda}_\r: \r^n \to M_\r$.
Let $\Lambda'_\r$ denote the kernel of $\tilde{\lambda}_\r$, and let
$S'$ be the semigroup $S \cap \Lambda'_\r$. 
Now consider the subalgebra $R' = R^\HH$, i.e., $R'$ is the $0$-graded
part of $R$ with respect to the $M$-grading. 
The following lemma follows readily from Lemmas~\ref{prop-S'} and~\ref{prop-tilde-lambda}. 

\begin{Lem} \label{prop-R'-S'}
The semigroup $S' := S \cap \Lambda'_\r$ coincides with the semigroup
$S(R')$ associated to the algebra $R'$ and the valuation
$\tilde{v}$. In particular, $S(R')$ is finitely generated.
\end{Lem}

\subsection{Integrable systems compatible with GIT and symplectic
  quotients}

We retain the notation of Section~\ref{subsec:invariant val}. 
In particular, $\HH$ denotes the torus which acts on $X$ (and
hence $R$). 
Since the valuation $\tilde{v}$ on $R$ is $\HH$-invariant by Lemma~\ref{prop-tilde-v-invariant}, 
the $M$-grading on $R$ is compatible with $\tilde{v}$ 
in the sense that for $k > 0$ and  $u \in \z^n$, the spaces $(R_k)_{\geq u}$ and $(R_k)_{> u}$ are $M$-graded. 
We record the following \cite[Proposition 5.18]{Anderson10}. 
Recall that $\RR$ is the family of algebras in Theorem \ref{th-toric-degen} which degenerates the coordinate ring $R$ to the coordinate ring of a semigroup algebra $\c[S]$.

\begin{Prop} \label{prop-lift-grading}
The $M$-grading on $R$ can be lifted to an $M$-grading on the family $\RR$.  
\end{Prop}

\begin{proof}[Sketch of proof]
We use notation as in the proof of Theorem
\ref{th-toric-degen}. Without loss of generality we may assume that
the elements
$f_{ij} \in R$ are homogeneous with respect to the $M$-grading. We
additionally extend the $M$-grading to $R[t]$ by defining 
$\deg(t) = 0$. Then the $\c[t]$-algebra $\RR$ generated by $t$ and the $\tilde{f}_{ij} = t^{w_{ij}} f_{ij}$ is an 
$(\n \times M \times \n)$-graded subalgebra of $R[t]$. In addition, we
can equip $\c[x_{ij}, \tau]$ with an $M$-grading by defining $\deg(\tau) = 0$ and 
$\deg(x_{ij}) = \deg(f_{ij}) = \tilde{\lambda}(i, u_{ij}) \in M$ where $\tilde{v}(f_{ij}) = (i, u_{ij})$. Then the map $\c[x_{ij}, \tau] \to \RR$
preserves the $M$-gradings. 
\end{proof}

Taking $\Proj$ we obtain the following. 

\begin{Cor} \label{cor-lift-action}
\begin{itemize}
\item[(a)] The $\HH$-action on $X$ lifts to an $\HH$-action on the 
  family $\X$.
\item[(b)] The family $\X$ is an $\HH$-invariant subvariety of $W\p \times \c$, where $\HH$ acts on $W\p$ via the natural action of
$\HH$ on $L^*, (L^2)^*, \ldots, (L^r)^*$, and $\HH$ acts on $\c$ trivially.
\end{itemize}
\end{Cor}

As mentioned above, in order to make the constructions of the toric
degeneration and the integrable system compatible with the
$\HH$-action, we must choose the Khovanskii basis $\{f_{ij}\}$ appropriately. More
specifically, we assume the following. 
\begin{itemize}
\item[(1)] Each $f_{ij}$ is homogeneous with respect to the $M$-grading.
\item[(2)] A subset of the collection $\{f_{ij}\}$ forms a Khovanskii basis 
  for $R'$. More precisely, 
For each $i$, $1 \leq i \leq r$, there exists $n'_i \leq n_i$ such
that the $\{f_{ij}\}_{1 \leq j \leq n'_i}$ are 
homogeneous of degree $0$ with respect to the $M$-grading, and the
collection $\{\tilde{v}(f_{ij})\}_{1 \leq i \leq r, 1 \leq j \leq n'_i}$ generate the semigroup $S(R')$.
\end{itemize}

Let $\RR$ and $\RR'$ denote the degenerating families corresponding to
$R$ and $R'$, respectively, constructed as in
Theorem~\ref{th-toric-degen}, using the collections $\{f_{ij}\}_{1
  \leq i \leq r, 1 \leq j \leq n_i}$ and $\{f_{ij}\}_{1 \leq i \leq r,
  1 \leq j \leq n'_i}$ respectively. The following is immediate. 

\begin{Lem}
The algebra $\RR'$ is the degree-$0$ part of $\RR$, i.e. $\RR' = \RR^\HH$, with respect to the
$M$-grading on $\RR$ from Proposition~\ref{prop-lift-grading}. 
\end{Lem}

Taking $\Proj$ we obtain the following. 

\begin{Cor} \label{cor-degen-commutes-with-GIT}
There exist flat families $\pi: \X \to \c$ and $\pi': \X' \to \c$ and a morphism $p: \X \to \X'$ such that:
\begin{itemize}
\item[(a)] The family $\X$ is an $\HH$-variety and the projection $\pi$
  is $\HH$-invariant. 
\item[(b)] The family $\X'$ is the GIT quotient $\X // \HH$ of $\X$ by
  $\HH$. 
\item[(c)] Let $\X^{ss} \subset \X$ denote the set of semistable
  points of $\X$ with respect to the $\HH$-action and let $p: \X^{ss}
  \to \X' = \X^{ss}/\HH$ be the quotient map. Then the diagram 
\begin{equation} \label{equ-diagram-commutative}
\xymatrix{
\X^{ss} \ar[rr]^{p} \ar[rd]_\pi & & \X' \ar[ld]^{\pi'} \\
& \c &\\
}
\end{equation}
commutes. 
\item[(vi)] 
The general fibers of $\X$ and $\X'$ are $X$ and $X' = X // \HH$,
respectively, and the special fibers are $X_0$ and $X'_0 = X_0 // \HH$,
respectively.  
\end{itemize}
\end{Cor}

Now let $\mu: X \to \r^n$ and $\mu': X' \to \r^{n-m}$ 
be the integrable systems constructed in Theorem \ref{th-A}
corresponding to the families $\X$ and $\X'$ respectively. The
following is now immediate from Theorem \ref{th-int-system-GIT-part1} and discussion in 
Section \ref{sec-GIT-part1}: 

\begin{Th}\label{th-int-system-GIT-part1-part2}
The following diagram is commutative:
\begin{equation} \label{equ-comm-diag-GIT}
\xymatrix{
& \Delta_H \ar@{}[r]|{\subset} & \Lie(H)^*\\
X \ar[r]^{\mu} \ar[ru]^{\mu_H} & \Delta \ar@{}[r]|{\subset} \ar[u]_{{\lambda}_\r} & \r^n \\
\mu_H^{-1}(0) \ar@{^{(}->}[u]^i \ar[d]_p & \\
X' \ar[r]^{\mu'} & \Delta' \ar@{}[r]|{\subset} \ar@{^{(}->}[uu] & \r^{n-m}\\
}
\end{equation}
\end{Th}

\section{Examples} \label{sec-examples}

In previous sections we used the notation $\HH$ to denote an
algebraic torus acting on the variety $X$ where $\dim_\c \HH = m$ is
possibly less than $n = \dim_\c X$, and we reserved the notation $\T$
to denote the algebraic torus whose dimension is precisely equal to
$\dim_\c X$. In the discussion below, we deviate from this
notation and use the notation $\T$ (as is standard in the literature) for the torus which acts
on $X$, even when $\dim_\c \T$ is strictly less than $\dim_\c X$. 

\subsection{Elliptic curves} 

Let $X$ be an elliptic curve, and $v$ be the valuation on $\c(X)$
associated to a point $p \in X$. Consider the line bundle $\lb =
\mathcal{O}_X(3p)$ and $L = H^0(X, \mathcal{O}_X(3p))$ giving the
cubic embedding of $X$ in $\c \p^2$. The semigroup $S(R(L)) \subset \n
\times \z$ is generated by $(1, 0)$, $(1, 1)$, $(1, 3)$ and hence
finitely generated \cite[Example 5]{Anderson10}. The curve $X$
degenerates to a cuspidal cubic curve. The Newton-Okounkov body
$\Delta(R(L))$ is the line segment $[0,3]$. Theorem \ref{th-B} gives a 
function $F: X \to \r$ which is continuous on all of
$X$, differentiable on a dense open subset $U$, with image precisely
$[0,3]$, {and whose Hamiltonian vector field generates a circle action
  on $U$. Here $U$ is the complement of the figure $8$ in $X$. 
See Figure~\ref{fig:elliptic}. Moreover, the inverse image of the open interval $(0,3)$ is
  contained in $U$}.

\begin{figure} 
\includegraphics{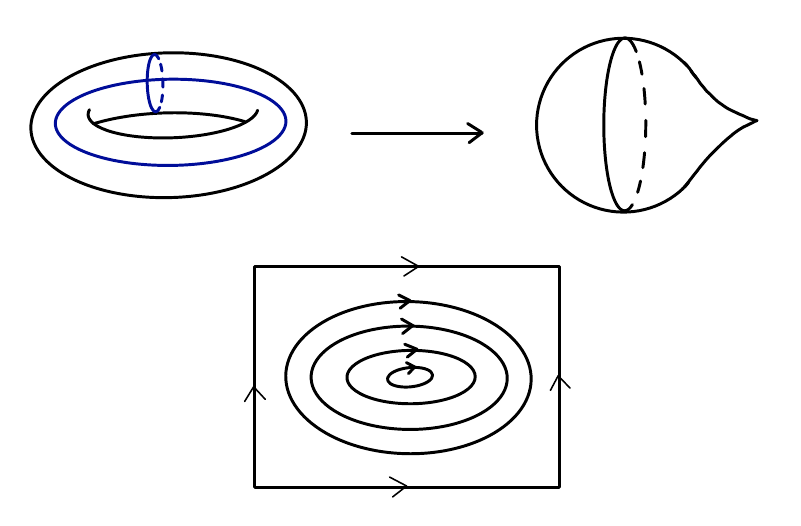} 
\caption{The elliptic curve degenerates to a cuspidal cubic curve
  (which is singular). The Hamiltonian 
  $S^1$-action on the smooth points of the cuspidal cubic is the
  usual rotation; this pulls back to an $S^1$-action on the complement
   of the figure $8$ in the elliptic curve.} 
\label{fig:elliptic} 
\end{figure}

\subsection{Flag varieties} \label{subsec-flag-var}

Let $G$ be a connected complex reductive algebraic group , $B$ a Borel subgroup and $\T$ its maximal torus.
We denote the weight lattice, i.e. the character lattice of $\T$, by $\Lambda$.
Then $\Lambda^+$ (respectively $\Lambda^+_\r$) are the semigroup of dominant weights (respectively positive
Weyl chamber) corresponding to the choice of $B$. We also fix a maximal compact subgroup $K$ compatible with the 
choice of $B$ and $\T$ so that $T = K \cap \T$ is a maximal torus
of $K$.

Let $G/B$ denote the complete flag variety of $G$. 
Given a regular dominant weight $\lambda$, i.e. a weight in the interior of 
the positive Weyl chamber, the variety $G/B$ embeds in the projective space 
$\p(V_\lambda)$ as the $G$-orbit of a highest weight vector. Here $V_\lambda$ is the irreducible $G$-module with highest weight $\lambda$. 
More generally, let $\lambda$ be a dominant weight (possibly on the boundary of the positive Weyl chamber). Then the $G$-orbit 
of a highest weight vector $v_\lambda$ in the projective space $\p(V_\lambda)$ is a partial flag variety 
$X_\lambda = G/P_\lambda$ where $P_\lambda$ is the
 $G$-stabilizer of $v_\lambda$ in $\p(V_\lambda)$. Let $\lb_\lambda$ be the restriction of the line bundle $\mathcal{O}(1)$ on the
 projective space $\p(V_\lambda)$ to $X_\lambda$. By the Borel-Weil-Bott theorem the space of sections $H^0(X_\lambda, \lb_\lambda)$ is
isomorphic to $V_\lambda^*$ as a $G$-module.

Let $N = \dim_\c(G/B)$. The so-called {\bf string polytope}
is a rational polytope in $\r^N$ such that the number of
integral points in the polytope is equal to $\dim_\c(V_\lambda)$. In
fact, more is true: the integral points in a string polytope
parameterize the so-called {\it crystal basis} for the $G$-module
$V_\lambda$ (see \cite{Littelmann}). The construction of the string polytope depends on the
choice of a \textit{reduced word decomposition $\s = (s_{i_1}, s_{i_2},
  \ldots, s_{i_N})$} for the longest element $w_0$ in the Weyl group
$W$ of $(G, T)$, where here the $s_i$ denote the simple reflections
corresponding to the simple roots $\alpha_i$. 
Thus, the string polytope associated to a dominant
weight $\lambda$ and a reduced decomposition $\s$ is often denoted as 
$\Delta_\s(\lambda)$ \cite{Bernstein-Zelevinsky01, Littelmann}.

\begin{Rem} \label{rem-GC}
The well-known Gel'fand-Cetlin polytopes \cite{G-C} corresponding to
irreducible representations of $\GL(n,\c)$ are special cases of the string
polytopes. More precisely, let $G=\GL(n, \c)$. Here the Weyl group
is $W = S_{n}$. Choose the reduced word decomposition 
$$w_0 = (s_1)(s_2s_1)(s_3s_2s_1) \cdots (s_{n-1}\cdots s_1)$$
for the longest element $w_0 \in S_n$, where $s_i$ denotes the simple transposition exchanging $i$ and
$i+1$. Then $\Delta_{\s}(\lambda)$ can be identified (after a linear change of coordinates) with the
Gel'fand-Cetlin polytope corresponding to $\lambda$. Similarly, for
$G = \SP(2n, \c)$ or $\SO(n, \c)$, for analogous well-chosen reduced
decompositions of the corresponding longest elements of the Weyl
group, we can recover the corresponding Gel'fand-Cetlin polytopes as
string polytopes \cite{Littelmann}.
\end{Rem}

In fact, given a reduced decomposition $\s$, 
there is a rational polyhedral cone $\mathcal{C}_\s$ in $\Lambda_\r^+ \times \r^N$ 
such that each string polytope $\Delta_\s(\lambda)$ is the slice of
the cone $\mathcal{C}_\s$ at $\lambda$, i.e., 
$\Delta_\s(\lambda) = \mathcal{C}_\s \cap \pi^{-1}(\lambda)$ where $\pi: \Lambda_\r^+ \times \r^N \to \Lambda_\r^+$ is the projection on the first factor \cite{Littelmann}.

We note the following. 
\begin{enumerate} 
\item We can define the polytope $\Delta_{\s}(\lambda) = \mathcal{C}_\s \cap \pi^{-1}(\lambda)$ for any $\lambda \in \Lambda_\r^+$.
\item The fact that $\mathcal{C}_\s$ is a convex cone implies that $\Delta_\s(k\lambda) =
k\Delta_\s(\lambda)$ for any $k > 0$. Moreover, for $\lambda_1$, $\lambda_2 \in \Lambda^+_\r$
we have $\Delta_\s(\lambda_1) + \Delta_\s(\lambda_2) \subset \Delta_\s(\lambda_1 + \lambda_2)$.
\item One also proves that the map:
\begin{equation} \label{equ-lambda-proj}
\pi_\lambda: (t_1, \ldots, t_N) \mapsto -\lambda + t_1\alpha_{i_1} + \cdots + t_N\alpha_{i_N},
\end{equation}
projects the string polytope $\Delta_\s(\lambda)$ onto the polytope $P(\lambda)$ which is the convex hull of the Weyl group
orbit of $\lambda$. 
\end{enumerate}

In \cite{Caldero} Caldero constructs a flat deformation of the flag variety $X_\lambda$
to the toric variety $X_{\lambda, w_0}$ corresponding to the string polytope $\Delta_\s(\lambda)$. 
The key ingredient in his construction is a multiplicativity property of the (dual) canonical basis
with respect to the string parametrization. In \cite{Kaveh-string} it is shown that Anderson's toric
degeneration recounted in Section~\ref{sec-toric-degen} is a
generalization of Caldero's construction. 

Fix a $K$-invariant Hermitian metric on $\p(V_\lambda)$. Let $\mu_T$ denote the moment map for the Hamiltonian 
action of the compact torus $T \subset K$ on $X_\lambda \cong G/P_\lambda$. Theorem \ref{th-B} 
applied to $(X_\lambda, \lb_\lambda)$ imply the following:
\begin{Cor} \label{cor-int-system-flag}
Let $n = \dim_\c(X_\lambda)$. There exists an integrable system $\mu_\lambda=(F_1, \ldots, F_n): X_\lambda \to \r^n$ 
(in the sense of Definition \ref{def-int-system}) such that
\begin{itemize}
\item[(a)] the image of $\mu_\lambda$ is the string polytope
  $\Delta_\w(\lambda)$, and 
\item[(b)] the diagram
$$
\xymatrix{
& P(\lambda) \ar@{}[r]|{\subset} & \Lambda_\r \\
X_\lambda \ar[r]_{\mu_\lambda} \ar[ru]^{\mu_T} & \Delta_\w(\lambda) \ar[u]_{\pi_\lambda} \ar@{}[r]|{\subset} & \r^n\\
}
$$
commutes, where the vertical arrow $\pi_\lambda$ is the linear projection given in \eqref{equ-lambda-proj}.\\
\item[(c)] {Moreover, the functions $F_1, \ldots, F_n$ generate a torus action on the open dense subset $U$, on which the $F_i$ are differentiable and Poisson commute, and the inverse image of the interior of $\Delta$ is 
contained in $U$.}
\end{itemize}  
\end{Cor}

\subsection{Spherical varieties}

In fact, the constructions and results in Section
\ref{subsec-flag-var} can be extended to the larger class of spherical
varieties. Spherical varieties are the algebraic analogues of 
multiplicity-free Hamiltonian spaces. More precisely, let $G$ be a
connected reductive algebraic group. Let $X$ be a {normal algebraic variety} equipped
with an algebraic $G$-action.
Then $X$ is called {\it spherical} if a (and hence any) Borel subgroup 
of $G$ has a dense open orbit. Flag varieties $G/P_\lambda$ and
$\T$-toric varieties are examples of spherical varieties, with respect
to the actions of $G$ and $\T$, respectively.

Let $X$ be a normal projective spherical $G$-variety of dimension $n$
and $\lb$ a $G$-linearized very ample line bundle. Then $X$ embeds
$G$-equivariantly in the projective space $\p(V)$ where $V = H^0(X,
\lb)^*$. Let $K$ denote a maximal compact subgroup of $G$ as in the
previous section, and fix a
$K$-invariant Hermitian product on $V$. This induces a K\"ahler metric
on $\p(V)$ and hence on the smooth locus of $X$. With respect to the
corresponding symplectic structure, (the smooth locus of) $X$ is a
Hamiltonian $K$-space; let $\mu_K$ denote the moment map. 
The Kirwan polytope, denoted $Q(X, \lb)$, is defined to be the
intersection of the moment map image of $\mu_K$ with the positive Weyl
chamber of $\Lie(K)^*$. 

For this discussion we fix a reduced decomposition $\w$ for the
longest element $w_0$. In \cite{Okounkov3} and \cite{Alexeev-Brion04}
the authors introduce a polytope $\Delta_\w(X, \lb) \subset
\Lambda_\r^+ \times \r^N$ defined as 
$$\Delta_\w(X, \lb) := \{ (\lambda, x) \mid \lambda \in Q(X, \lb),~ x \in \Delta_\w(\lambda) \}.$$
That is, $\Delta_\w(X, \lb)$ is the polytope fibered over the moment polytope $Q(X, \lb)$ with 
the string polytopes as fibers. 
We call the polytope $\Delta_\w(X, \lb)$ the {\it string polytope of the spherical variety $X$}. In \cite{Alexeev-Brion04} and 
\cite{Kaveh-SAGBI} it is shown that there is a flat degeneration of $X$ to the toric variety associated to the 
rational polytope $\Delta_\w(X, \lb)$. In fact, in \cite{Kaveh-string}
it is shown that, with respect to certain choices of valuations, 
the string polytopes $\Delta_\w(\lambda)$ and $\Delta_\w(X, \lb)$ 
can be realized as Newton-Okounkov bodies, and the degenertions in \cite{Caldero}, \cite{Alexeev-Brion04}  and \cite{Kaveh-SAGBI} all fit into the general toric degeneration framework discussed in Section \ref{sec-toric-degen}.

Generalizing the case of the flag varieties discussed in
Section~\ref{subsec-flag-var}, we have the following. Let $P(X, \lb)$
denote the polytope which is the convex hull of the $W$-orbit of the
Kirwan polytope $Q(X,\lb)$. It can be shown that $P(X,\lb)$ is
precisely the moment polytope for the Hamiltonian $T$-action on
$X$. Let $\mu_T$ denote the moment map for this action. 
Applying Theorem \ref{th-B} to $(X, \lb)$, we obtain the
following.

\begin{Th} \label{cor-int-system-spherical}
Let $X$ and $\lb$ be as above. Let $\dim_\c X = n$. Then there exists an integrable system $\mu=(F_1, \ldots, F_n): X \to \r^n$
(in the sense of Definition \ref{def-int-system}) such that
\begin{itemize}
\item[(a)] the image of $\mu$ can be identified with the string
  polytope $\Delta_\w(X, \lb)$, and 
\item[(b)] the diagram
$$
\xymatrix{
& P(X, \lb)\\
X \ar[r]_{\mu} \ar[ru]^{\mu_T} & \Delta_\w(X, \lb) \ar[u]_{\pi}\\
}
$$
commutes, where the vertical arrow $\pi$ is the linear projection given in \eqref{equ-lambda-proj}.
\item[(c)] {Moreover, the functions $F_1, \ldots, F_n$ generate a torus action on the open dense subset $U$, on which the $F_i$ are differentiable and Poisson commute, and the inverse image of the interior of $\Delta$ is 
contained in $U$.}
\end{itemize}  
\end{Th}

\begin{Rem}
Let $X$ be a smooth spherical $G$-variety of dimension $n$ equipped with a $G$-linearized very ample line bundle $\lb$.
It can be shown that the space of smooth $K$-invariant functions on
$X$ is commutative, in the sense that any two smooth $K$-invariant
functions Poisson-commute. We believe that the 
integrable system $\mu = (F_1, \ldots, F_n)$ in Theorem~\ref{cor-int-system-spherical} can be 
constructed so that $(F_1, \ldots, F_r)$ are smooth $K$-invariant
functions on all of $X$, where $r \leq n$ is the so-called 
{\it rank} of the spherical variety (i.e., the minimal codimension of a $K$-orbit).
\end{Rem}

\subsection{Weight varieties}

We maintain the notation of Section \ref{subsec-flag-var}. Let
$P(\lambda)$ denote the convex hull of the $W$-orbit of $\lambda$.  
Recall that a \textbf{weight variety} $X_{\lambda, \gamma}$ is the GIT quotient of $X_\lambda$ by the action of $\T$ 
twisted by an integral weight $\gamma \in P(\lambda) \cap \Lambda^+$. 
More precisely, let $\gamma \in 
P(\lambda) \cap \Lambda^+$ be a character of $\T$ and $\lb_\lambda(-\gamma)$ be the $\T$-line bundle $\lb_\lambda$
where the action of $\T$ is twisted by $-\gamma$. Then one defines 
$$X_{\lambda, \gamma} := X_\lambda^{ss} (\lb_\lambda(-\gamma)) / \T.$$
Important examples of weight varieties are polygon spaces \cite{HausmannKnutson:1997}. 
 
An alternative description is as follows. Let $V_\lambda^{(\gamma)}$
denote the $\gamma$-weight space in the $G$-module $V_\lambda$. The
weight variety $X_{\lambda, \gamma}$ is $\Proj(R_\lambda^{(\gamma)})$
where
$$R_\lambda^{(\gamma)} := \bigoplus_{k} (V_{k\lambda}^*)^{(k\gamma)}.$$
In other words, $R_\lambda^{(\gamma)}$ is the $\T$-invariant subalgebra of $R(L_\lambda)$ for the $(-\gamma)$-twisted
action of $\T$ on $R(L_\lambda)$, defined by
$$t *_\gamma f := \gamma(t)^{-k} (t \cdot f) \quad \forall t \in \T,~ \forall f \in L^k.$$
When $\gamma$ is a regular value for the moment map $\mu_T$, and $T$
acts freely on $\mu_T^{-1}(0)$, 
the variety $X_{\lambda, \gamma}$ can also be identified with the symplectic quotient of $X_\lambda$ at the value $\gamma$. 

Using Theorem~\ref{cor-degen-commutes-with-GIT} we can now recover the
following theorem of Foth and Hu \cite{Foth-Hu}.

\begin{Th} \label{th-Foth-Hu}
There exists a flat degeneration of the weight variety $X_{\lambda, \gamma}$ to a projective toric variety 
$X_{\lambda, \gamma, 0}$ corresponding to the polytope
$$\Delta_\w(\lambda, \gamma) = \Delta_\w(\lambda) \cap \pi_\lambda^{-1}(\gamma),$$
obtained by slicing the string polytope $\Delta_\w(\lambda)$ at $\gamma$. Here the projection $\pi_\lambda$ 
is that given in \eqref{equ-lambda-proj}.
\end{Th}

As in Section \ref{sec-GIT-part1}, the GIT quotient $X_{\lambda, \gamma}$ inherits a K\"ahler structure from $X_\lambda$. 
Let $n' = \dim_\c X_{\lambda, \gamma}$. From 
Corollary \ref{cor-int-system-flag} and Theorem \ref{th-int-system-GIT-part1-part2} we
obtain the following. 

\begin{Cor}
Under the assumptions and notation as above, there exists an integrable system $\mu_{\lambda, \gamma} = (F_1, \ldots, F_{n'})$ 
on $X_{\lambda, \gamma}$ (in the sense of Definition \ref{def-int-system}) such that
\begin{itemize}
\item[(a)] the image of $\mu_{\lambda, \gamma}$ is the string polytope
  $\Delta_\w(\lambda, \gamma)$, and 
\item[(b)] the diagram
\begin{equation} \label{equ-comm-diag-weight-var}
\xymatrix{
& P(\lambda) \ar@{}[r]|{\subset} & \Lie(T)^*\\
X_\lambda \ar[r]^{\mu_\lambda} \ar[ru]^{\mu_T} & \Delta_\w(\lambda) \ar@{}[r]|{\subset} 
\ar[u]_{\pi_\lambda} & \r^n \\
\mu_T^{-1}(\gamma) \ar@{^{(}->}[u]^i \ar[d]^p & \\
X_{\lambda, \gamma} \ar[r]^{\mu_{\lambda, \gamma}} & \Delta_\w(\lambda, \gamma) 
\ar@{}[r]|{\subset} \ar@{^{(}->}[uu] & \r^{n'}\\
}
\end{equation}
commutes, where the map $\r^{n'} \hookrightarrow \r^{n}$ is the inclusion of the first $n'$ coordinates.
\end{itemize}
\end{Cor}

\def\cprime{$'$}

\end{document}